\documentclass[12pt]{amsart}

\usepackage{amsmath,amssymb, amsthm}
\usepackage[margin=2.5cm]{geometry}
\usepackage{graphicx,psfrag}
\usepackage{ifthen}
\usepackage{verbatim}

\usepackage[dvips]{color}
\def\dpr#1{{#1}}

\def\R{{\mathbb R}}
\def\N{{\mathbb N}}
\def\AA{{\mathcal A}}
\def\EE{{\mathcal E}}
\def\HH{{\mathcal H}}
\def\JJ{{\mathcal J}}
\def\MM{{\mathcal M}}
\def\OO{{\mathcal O}}
\def\SS{{\mathcal S}}
\def\TT{{\mathcal T}}
\def\LL{{\mathcal L}}
\def\PP{{\mathcal P}}

\def\NN{{\mathcal N}}
\def\hone{H^1(\Omega)}

\def\KK{\mathcal K}
\def\KKell{\KK_\ell^\star}

\def\edual#1#2{\langle\hspace*{-1mm}\langle#1\,,\,#2\rangle\hspace*{-1mm}\rangle}
\def\diam{{\rm diam}}
\def\norm#1#2{\|#1\|_{#2}}

\def\enorm#1{|\hspace*{-.5mm}|\hspace*{-.5mm}|#1|\hspace*{-.5mm}|\hspace*{-.5mm}|}
\def\set#1#2{\big\{#1\,:\,#2\big\}}
\def\eps{\varepsilon}
\def\dual#1#2{\langle#1\,,\,#2\rangle}

\def\oscl1#1#2{\operatorname{osc_{\ell+1}}(#1|#2)^2}

\def\newset#1#2{\{#1 \hspace{1mm}:\hspace{1mm} #2\}}

\newcounter{constantsnumber}
\def\setc#1{\refstepcounter{constantsnumber}%
\label{const#1}C_{\theconstantsnumber}}
\def\c#1{C_{\ref{const#1}}}
\def\osc{{\rm osc}}
\def\conv{{\rm conv}}
\def\apx{{\rm apx}}

\newtheorem{theorem}{Theorem}
\newtheorem{proposition}[theorem]{Proposition}
\newtheorem{lemma}[theorem]{Lemma}

\newtheorem{algorithm}[theorem]{Algorithm}
\newenvironment{remark}{\paragraph{Remark.}\ \it}{\qed\bigskip}
\newcounter{step}
\renewenvironment{proof}[1][]%
{\setcounter{step}{0}\medskip \ifthenelse{\equal{#1}{}}{\emph{Proof.~}}{\emph{#1.~}}}%
{\qed\bigskip}

\newcommand{\myfootnote}[1]{%
\renewcommand{\thefootnote}{}%
\footnotetext{#1}%
\renewcommand{\thefootnote}{\arabic{footnote}}%
}



\def\subsection#1
{
\bigskip

\refstepcounter{subsection}
{\bf\arabic{section}.\arabic{subsection}.~#1.~}
}

\begin{document}

\title[Convergence of AFEM for obstacle problem]{Convergence of Adaptive FEM\\ for Some Elliptic Obstacle Problem\\ with Inhomogeneous Dirichlet Data}
\date{}

\author{M. Feischl}

\author{M. Page}

\author{D. Praetorius}
\address{Institute for Analysis and Scientific Computing,
     Vienna University of Technology,
     Wiedner Hauptstra\ss{}e 8-10,
     A-1040 Wien, Austria}
\email{\{\,Michael.Feischl\,,\,Dirk.Praetorius\,\}@tuwien.ac.at}
\email{Marcus.Page@tuwien.ac.at\quad\rm(corresponding author)}


\maketitle

\begin{abstract}
In this work, we show the convergence of adaptive lowest-order FEM (AFEM) for an elliptic obstacle problem with non-homogeneous Dirichlet data,
where the obstacle $\chi$ is restricted only by $\chi \in H^2(\Omega)$.
The adaptive loop is steered by some residual based error estimator
introduced in \textsc{Braess, Carstensen \& Hoppe} (2007) that is extended to control oscillations of the Dirichlet data, as well. 
In the spirit
of \textsc{Cascon et al}.\ (2008), we show that a weighted sum of energy error, estimator, and Dirichlet oscillations satisfies a 
contraction
property up to certain vanishing energy contributions. This result extends the analysis of \textsc{Braess, Carstensen \& Hoppe} (2007) and
\textsc{Page \& Praetorius} (2010) to the
case of non-homogeneous Dirichlet data as well as certain non-affine obstacles and introduces some energy estimates to overcome the lack of nestedness of the
discrete spaces.
\end{abstract}


\section{Introduction}
\subsection{Comments on prior work}
\myfootnote{\textit{Key words and phrases.} Adaptive finite element methods, Elliptic obstacle problems, Convergence analysis.}
\myfootnote{2000 \textit{Mathematics Subject Classification.} 65N30, 65N50.}
Adaptive finite element methods based on various types of a posteriori error estimators are a famous tool in science and engineering 
and are used to deal with a wide range of problems. As far as elliptic boundary value problems are concerned, convergence and even 
quasi-optimality of the adaptive scheme is well understood and analyzed, see e.g.\ 
\cite{ao, ckns, d1996, mns, msv, stevenson, verfuerth}.

In recent years the analysis has been extended and adapted to cover more general applications, such as the $p$-Laplacian \cite{veeser},
mixed methods~\cite{ch06:1}, non-conforming elements~\cite{ch06:2}, and obstacle problems. The latter is a classic introductory example
to study variational inequalities which represent a whole class of problems that often arise in physical \dpr{and economical context. One
major application is the oscillation of a membrane that must stay above a certain obstacle. Other examples are
filtration in porous media or the Stefan problem (i.e.\ melting solids). In both of which, non-homogeneous Dirichlet data play an important role.
Also in the financial world, obstacle problems arise, e.g.\ in the valuation of the American put option \cite{ps}, where one
has to deal with various non-affine obstacles.} For a broader understanding of these problems, we refer to~\cite{fr1982} and the references therein.
\dpr{The great applicability in many scientific areas thus make numerical analysis and mathematical understanding of the obstacle problem both, interesting and important.}
As far as a posteriori error analysis is concerned, we refer to~\cite{aol, b2005, bc, cn, ly, nsv, v2001}. Convergence of an 
adaptive method for elliptic obstacle problems with globally affine obstacle was proven in~\cite{cbh2007, pp2010}. Both of these 
works, however, considered homogeneous Dirichlet boundary data and affine obstacles only.
\subsection{Contributions of current work}
\dpr{We treat the case of a general obstacle $\chi \in H^2(\Omega)$. By a simple transformation and allowing non-homogeneous Dirichlet data (Prop.~\ref{prop:zero_obstacle}), this can, however,
be reduced to the case of a constant zero-obstacle. Since our analysis works for general globally affine obstacles, even without
the reduction step, we consider affine obstacles and non-homogeneous Dirichlet data in the following.} We follow the ideas from~\cite{pp2010}, i.e.\ adaptive P1-FEM for some elliptic obstacle problem with globally affine obstacle. Contrary to~\cite{pp2010} and~\cite{cbh2007}, however, we allow non-homogeneous Dirichlet boundary data $g\in H^1(\Gamma)$, which are approximated by $g_\ell$ via nodal interpolation within each step of the adaptive loop. In contrast to the aforementioned works, we thus do not have nestedness of the discrete ansatz sets, which is a crucial ingredient of the prior convergence proofs. In the spirit of~\cite{ckns} and in analogy to~\cite{pp2010}, we show that our adaptive algorithm, steered by some estimator $\varrho_\ell$, guarantees that the combined error quantity
\begin{align}
\Delta_\ell := \JJ(U_\ell) - \JJ(u_\ell) + \gamma\,\varrho_\ell^2 + \lambda\,\apx_\ell^2
\end{align}
is a contraction up to some vanishing perturbations $\alpha_\ell \rightarrow 0$, i.e.\
\begin{align}
\Delta_{\ell+1} \le \, \kappa \Delta_\ell + \alpha_\ell,
\end{align}
with $0 < \gamma, \kappa < 1, \lambda > 0$, and $\alpha_\ell \ge 0$.
The data oscillations on the Dirichlet boundary are controlled by the term $\apx_\ell$, and the quantity $u_\ell$ denotes the continuous solution subject to discrete boundary data $g_\ell$, which is introduced to circumvent the lack of nestedness of the discrete spaces. Convergence then follows from a weak
reliability estimate of $\varrho_\ell$, namely
\begin{align}
\varrho_\ell \rightarrow 0
\quad \Rightarrow \quad
\norm{u - U_\ell}{H^1(\Omega)} \rightarrow 0,
\end{align}
since $\varrho_\ell\lesssim\Delta_\ell\to0$ as $\ell\to\infty$.
We point out that our convergence proof makes use of the so called \emph{estimator reduction} and does therefore not need the \emph{interior node property}, which makes the result fairly independent of the local mesh-refinement strategy.

\subsection{Outline of current work}
In Section~\ref{section:continuous}, we formulate the continuous model problem and recall its unique solvability. In 
Section~\ref{section:discrete}, the same is done for the discretized problem. Section~\ref{section:algorithm} is a collection of the 
main results of this paper. Here, we introduce the error estimator $\varrho_\ell$, which is a generalization of the corresponding 
estimators from~\cite{cbh2007, pp2010}. We then state its weak reliability (Theorem \ref{thm:reliability}) and our version of the 
adaptive algorithm (Algorithm \ref{algorithm}). Finally (Theorem \ref{thm:convergence}), we state that the sequence of discrete 
solutions indeed converges towards the continuous solution $u \in H^1(\Omega)$. The subsequent 
Sections~\ref{section:reliability}--\ref{section:numerics} are then devoted to the proofs of the aforementioned results and numerical 
illustrations.

\section{Model Problem}
\label{section:continuous}%

\subsection{Problem formulation}
\label{section:continuous:problem}%
We consider an elliptic obstacle problem in $\R^2$ on a bounded domain
$\Omega$ with polygonal boundary $\Gamma:=\partial \Omega$. An obstacle on
$\overline\Omega$ is defined by the smooth function $\chi\in H^2(\Omega)$. Moreover, we consider inhomogeneous Dirichlet data $g \in H^{1/2}(\Gamma)$ and thus additionally require $\chi \le g$ almost everywhere on $\Gamma$. By
\begin{align}\label{eq:A}
 \AA := \newset{v \in \hone}{v\ge \chi \text{ a.e.\ in } \Omega, \, v|_\Gamma = g\text{ a.e.\ on }\Gamma},
\end{align}
we denote the set of admissible functions.
For given $f \in L^2(\Omega)$, we consider the energy functional
\begin{align}\label{eq:ourfunctional}
 \JJ(v) = \frac12 \edual{v}{v} - (f,v),
\end{align}
with the bilinear form
\begin{align}\label{eq:edual}
\edual{u}{v} = \int_\Omega\nabla u\cdot\nabla v\,dx
\quad\text{for all }u,v\in \hone
\end{align}
and with the $L^2$-scalar product
\begin{align}\label{eq:L2}
(f,v) = \int_\Omega fv\,dx.
\end{align}
By $\enorm\cdot$, we denote the energy norm on $H^1_0(\Omega)$ induced by $\edual\cdot\cdot$.
The obstacle problem then reads as follows: \emph{Find $u\in \AA$ such that}
\begin{align}\label{eq:obstacle}
\JJ(u) = \min_{v\in\AA}\JJ(v).
\end{align}

\subsection{Unique solvability}
\label{section:continuous:solvability}%
For the sake of completeness and to collect the main arguments also needed below, we recall the proof that the obstacle problem~\eqref{eq:obstacle} admits a unique solution. We stress that the following argument holds for any measurable obstacle $\chi$ with meaningful trace $\chi|_\Gamma$. Our restriction to smooth obstacles $\chi\in H^2(\Omega)$ is needed for the equivalent reformulation in Section~\ref{section:continuous:reformulation}.

\begin{proposition}\label{prop:continuous:solvability}
For given data $(\chi,g,f)\in H^2(\Omega)\times H^{1/2}(\Gamma)\times L^2(\Omega)$, the obstacle problem~\eqref{eq:obstacle} admits a unique solution $u\in\AA$
which is
equivalently characterized by the variational inequality
\begin{align}\label{eq:var_in}
\edual{u}{v-u} \ge (f,v-u)
\end{align}
for all $v \in \AA$.\qed
\end{proposition}

The following two lemmata provide the essential ingredients to prove Proposition~\ref{prop:continuous:solvability}. We start with a well-known abstract result,
cf.\ e.g.~\cite[Section~2.4--2.6]{b2007}.

\begin{lemma}\label{lem:minimization}
Let $\HH$ be a Hilbert space with scalar product $\edual\cdot\cdot$ and $\KK \subseteq \HH$ be a closed, convex, and
 non-empty subset. Then, for given $L \in \HH^*$, the variational problem
\begin{align}\label{eq:minimization}
 \JJ(u) = \min_{v \in\KK} \JJ(v)
 \quad\text{with}\quad
 \JJ(v) = \frac{1}{2}\edual{v}{v} - \dual{L}{v}
\end{align}
has a unique solution $u \in \KK$, where $\dual{\cdot}{\cdot}$ denotes the dual pairing
between $\HH$ and $\HH^*$. In addition, this solution is equivalently
characterized by the variational inequality
\begin{align}
\edual{u}{v-u} \ge \dual{L}{v-u}
\end{align}
for all $v \in \KK$.
\qed
\end{lemma}

To apply Lemma~\ref{lem:minimization}, we observe that the obstacle problem~\eqref{eq:obstacle} can be shifted into a setting with homogeneous Dirichlet data and $\HH=H^1_0(\Omega)$.
This involves a standard lifting operator $\LL:H^{1/2}(\Gamma)\to H^1(\Omega)$, see e.g.~\cite[Theorem 3.37]{m2000}, with the properties
\begin{align}
 (\LL v)|_\Gamma = v
 \quad\text{and}\quad
 \norm{\LL v}{H^1(\Omega)} \le \c{lifting}\,\norm{v}{H^{1/2}(\Gamma)},
\end{align}
where the constant $\setc{lifting}>0$ depends only on $\Omega$.

With elementary algebraic manipulations,
see e.g.~\cite[Section II.6]{k},
one obtains the following well-known
link between the obstacle problem~\eqref{eq:obstacle} and the abstract
minimization problem~\eqref{eq:minimization} from Lemma~\ref{lem:minimization}.

\begin{lemma}\label{lem:zeroboundary}
Let $\widehat g\in H^1(\Omega)$ be an arbitrary extension of
$g\in H^{1/2}(\Gamma)$, e.g., $\widehat g=\LL g$.
For $u\in\AA$ and $u - \widehat g = \widetilde u \in \KK
:=\newset{\widetilde v \in H^1_0(\Omega)}{\widetilde v \ge \chi - \widehat g}$,
the following statements are equivalent:
\begin{itemize}
\item[(i)] $u$ solves the obstacle problem~\eqref{eq:obstacle}.
\item[(ii)] $u$ satisfies $\edual{u}{v-u} \ge (f, v- u)$ for all $v\in\AA$.
\item[(iii)] $\widetilde u$ satisfies
 $\edual{\widetilde u}{\widetilde v - \widetilde u}
 \ge (f, \widetilde v - \widetilde u)
 - \edual{\widehat g}{\widetilde v - \widetilde u}$
for all $\widetilde v\in \KK$.\qed
\end{itemize}
\end{lemma}

\subsection{Reduction to problem with zero obstacle}
\label{section:continuous:reformulation}%
The following proposition allows to restrict to obstacle problems~\eqref{eq:obstacle}, where the obstacle $\chi$ is even zero. This provides the formulation which can be treated by our adaptive method below.

\begin{proposition}\label{prop:zero_obstacle}
For some smooth obstacle $\chi\in H^2(\Omega)$, the obstacle problem~\eqref{eq:obstacle} with data $(\chi,g,f)$ is equivalent to the obstacle problem with data $(0,g-\chi|_\Gamma,f+\Delta\chi)$. If $u\in H^1(\Omega)$ solves the obstacle problem with data $(0,g-\chi|_\Gamma,f+\Delta\chi)$, $u+\chi$ is the unique solution with respect to the data $(\chi,g,f)$.
\end{proposition}

\begin{proof}
The solution $u\in\AA$ of the obstacle problem with data $(\chi,g,f)$ is characterized by
\begin{align*}
 \edual{u}{v-u} \ge (f, v-u)
\end{align*}
for all $v\in\AA = \newset{v \in H^1(\Omega)}{v \ge \chi \text{ a.e.\ in } \Omega, v|_\Gamma = g}$.
Substitution $\widetilde u := u-\chi$ and $\widetilde v := v-\chi$ shows that this is equivalent to
\begin{align*}
 \edual{\widetilde u}{\widetilde v-\widetilde u} \ge (f,\widetilde v-\widetilde u) - \edual{\chi}{\widetilde v-\widetilde u}
 \quad\text{for all }
 \widetilde v\in\widetilde \AA,
\end{align*}
where $\widetilde\AA :=
\newset{\widetilde v \in H^1(\Omega)}{\widetilde v \ge 0\text{ a.e.\ in } \Omega, v|_\Gamma = g-\chi|_\Gamma}$. Finally, integration by parts with $w:=\widetilde v-\widetilde u\in H^1_0(\Omega)$ proves
\begin{align*}
 - \edual{\chi}{\widetilde v-\widetilde u}
 = -\int_\Omega\nabla\chi\cdot\nabla w\,dx
 = \int_\Omega\Delta\chi\,w\,dx.
\end{align*}
This concludes the proof.
\end{proof}

\subsection{Model problem}
\label{section:model:problem}%
According to the observation of Proposition~\ref{prop:zero_obstacle}, we may restrict to the case $\chi=0$ in the following. Having obtained a FE approximation $U_\ell$ 
of $u$ for the zero obstacle case, we may simply consider $U_\ell+\chi$ to obtain an approximation of the original problem with 
obstacle $\chi\in H^2(\Omega)$. \dpr{Since our analysis directly covers affine obstacles, we shall allow that $\chi$ is
globally affine on $\overline \Omega$, i.e.\ we consider Problem~\eqref{eq:obstacle} with respect to the data $(\chi, g, f) \in \PP^1(\overline\Omega)\times H^{1/2}(\Gamma)\times L^2(\Omega)$.} 

\section{Galerkin Discretization}
\label{section:discrete}%

\subsection{Problem formulation}
\label{section:discrete:problem}%
For the numerical solution of \eqref{eq:obstacle} by an adaptive finite
element method, we consider conforming and in the sense of Ciarlet regular
triangulations $\TT_\ell$ of $\Omega$ and denote the standard P1-FEM space
of globally continuous and piecewise affine functions by $\SS^1(\TT_\ell)$. Note that
a discrete function $V_\ell\in\SS^1(\TT_\ell)$ cannot satisfy the
inhomogeneous Dirichlet conditions $g\in H^{1/2}(\Gamma)$ in general. We therefore have to approximate $g\approx g_\ell \in \SS^1(\TT_\ell|_\Gamma)$,
 where the space $\SS^1(\TT_\ell|_\Gamma)$ denotes the space of globally continuous and piecewise affine functions on the
boundary $\Gamma$.
For this discretization, we assume additional regularity $g\in H^1(\Gamma)$ and
consider the approximation $g_\ell \in \SS^1(\TT_\ell|_\Gamma)$ which is
derived by nodal interpolation of the boundary data. Note that nodal
interpolation is well-defined since the Sobolev inequality on the 1D
manifold $\Gamma$ predicts the continuous inclusion
$H^1(\Gamma)\subset C(\Gamma)$. Altogether, the set of
discrete admissible functions $\AA_\ell$ is given by
\begin{align}\label{eq:Aell}
 \AA_\ell := \newset{V_\ell \in \SS^1(\TT_\ell)}{V_\ell \ge 0 \text{ in }\Omega, \, V_\ell = g_\ell \text{ on }\Gamma},
\end{align}
and the discrete minimization problem reads: \emph{Find $U_\ell\in \AA_\ell$ such that}
\begin{align}\label{eq:obstacle_discrete}
\JJ(U_\ell) = \min_{V_\ell\in \AA_\ell}\JJ(V_\ell).
\end{align}

\begin{remark}
Note that $0\le g\in H^1(\Gamma)\subset C(\Gamma)$
implies that $0 \le g(z) = g_\ell(z)$ for all nodes $z\in\Gamma$.
Therefore, we conclude $0 \le g_\ell$ on $\Gamma$
for the nodal interpolant $g_\ell\in\SS^1(\TT_\ell|_\Gamma)$.
\end{remark}
%

\subsection{Notation}
\label{section:discrete:notation}%
From now on, $\NN_\ell$ denotes the set of nodes of the regular triangulation $\TT_\ell$.
The set of all interior edges $E=T^+\cap T^-$ for certain elements $T^+,T^-\in\TT_\ell$ is
denoted by $\EE_\ell^\Omega$. The set of all edges of $\TT_\ell$ is denoted by $\EE_\ell$. In particular,
$\EE_\ell|_\Gamma := \EE_{\ell}^\Gamma:=\EE_\ell\backslash\EE_\ell^\Omega$ contains all boundary edges and provides some partition of $\Gamma$.

We recall that $\SS^1(\TT_\ell) = \set{V_\ell\in C(\overline\Omega)}{V_\ell|_T\text{ affine for all }T\in\TT_\ell}$ denotes the conforming P1-finite element space. Moreover,
$\SS^1_0(\TT_\ell) = \SS^1(\TT_\ell)\cap H^1_0(\Omega) = \set{V_\ell\in\SS^1(\TT_\ell)}{V_\ell|_\Gamma=0}$.

\subsection{Unique solvability}
\label{section:discrete:solvability}%
In this section, we recall that the discrete obstacle problem~\eqref{eq:obstacle_discrete} admits a unique solution which is again characterized by a variational inequality.

\begin{proposition}\label{prop:discrete:solvability}
The discrete obstacle problem~\eqref{eq:obstacle_discrete} admits a unique
solution $U_\ell \in \AA_\ell$, which is equivalently characterized
by the variational inequality
\begin{align}\label{eq:obstacle_discrete:inequality}
 \edual{U_\ell}{V_\ell - U_\ell}\ge \dual{f}{V_\ell - U_\ell}
\end{align}
for all $V_\ell\in\AA_\ell$.\qed
\end{proposition}

The proof of Proposition~\ref{prop:discrete:solvability} is obtained as in the continuous case. It relies on the fact that discrete boundary data $g_\ell\in\SS^1(\TT_\ell|_\Gamma)$ can be lifted to a discrete function $\widehat g_\ell\in\SS^1(\TT_\ell)$ with $\widehat g_\ell|_\Gamma = g_\ell$. The proof of the latter is a consequence of (and even equivalent to, see~\cite{ds2003}) the existence of the Scott-Zhang quasi-interpolation operator $P_\ell:H^1(\Omega)\to\SS^1(\TT_\ell)$, see~\cite{sz}, which is a linear and continuous projection onto $\SS^1(\TT_\ell)$, i.e.\
\begin{align}\label{eq:sz:cont}
 P_\ell^2 = P_\ell
 \quad\text{and}\quad
 \norm{P_\ell v}{H^1(\Omega)} \le \c{sz:stetig} \norm{v}{H^1(\Omega)},
 \quad\text{for all }v\in H^1(\Omega),
\end{align}
that preserves discrete boundary conditions, i.e.
\begin{align}
 (P_\ell v)|_\Gamma = v|_\Gamma
 \quad
 \text{for all }v\in H^1(\Omega)
 \text{ with }v|_\Gamma \in \SS^1(\TT_\ell|_\Gamma).
\end{align}
The constant $\setc{sz:stetig} > 0$ depends only on the shape regularity constant $\sigma(\TT_\ell)$ and on the diameter of $\Omega$. Moreover, $P_\ell$ has a local first-order approximation property which is, however, not used throughout.

\section{Adaptive Mesh-Refining Algorithm and Main Results}
\label{section:algorithm}%

\begin{figure}[t]
\centering
\psfrag{T0}{}
\psfrag{T1}{}
\psfrag{T2}{}
\psfrag{T3}{}
\psfrag{T4}{}
\psfrag{T12}{}
\psfrag{T34}{}
\includegraphics[width=30mm]{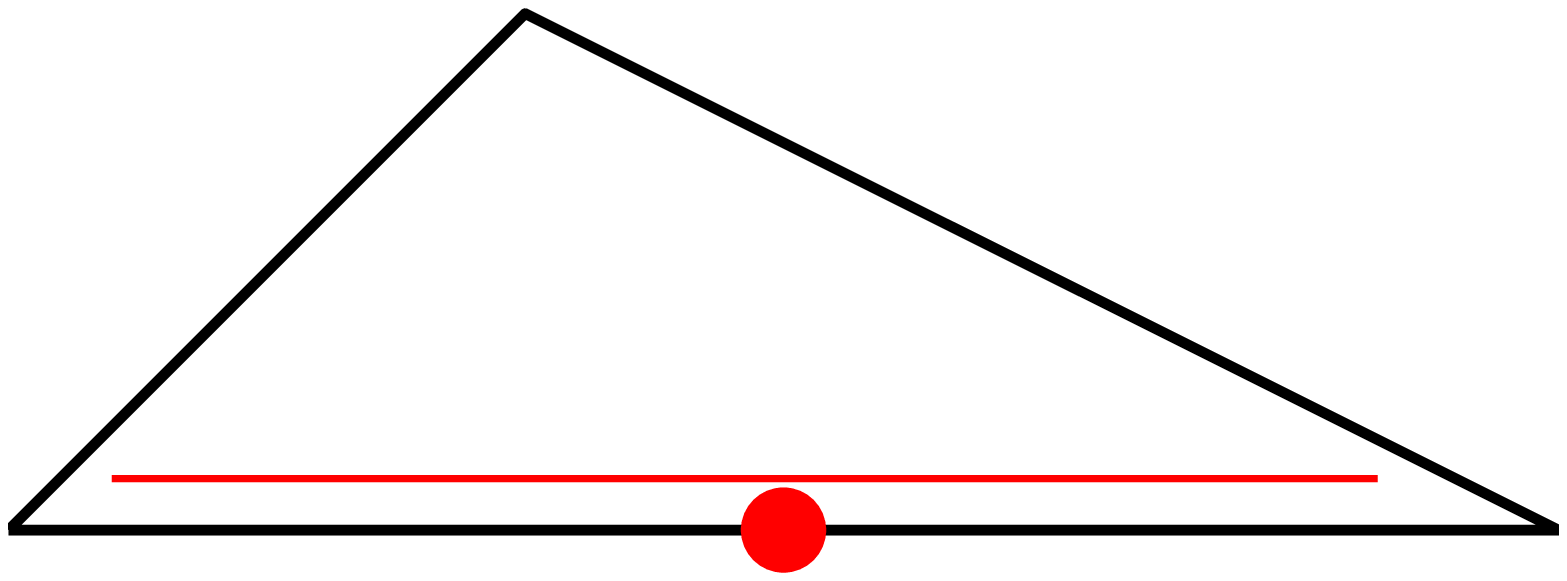} \quad
\includegraphics[width=30mm]{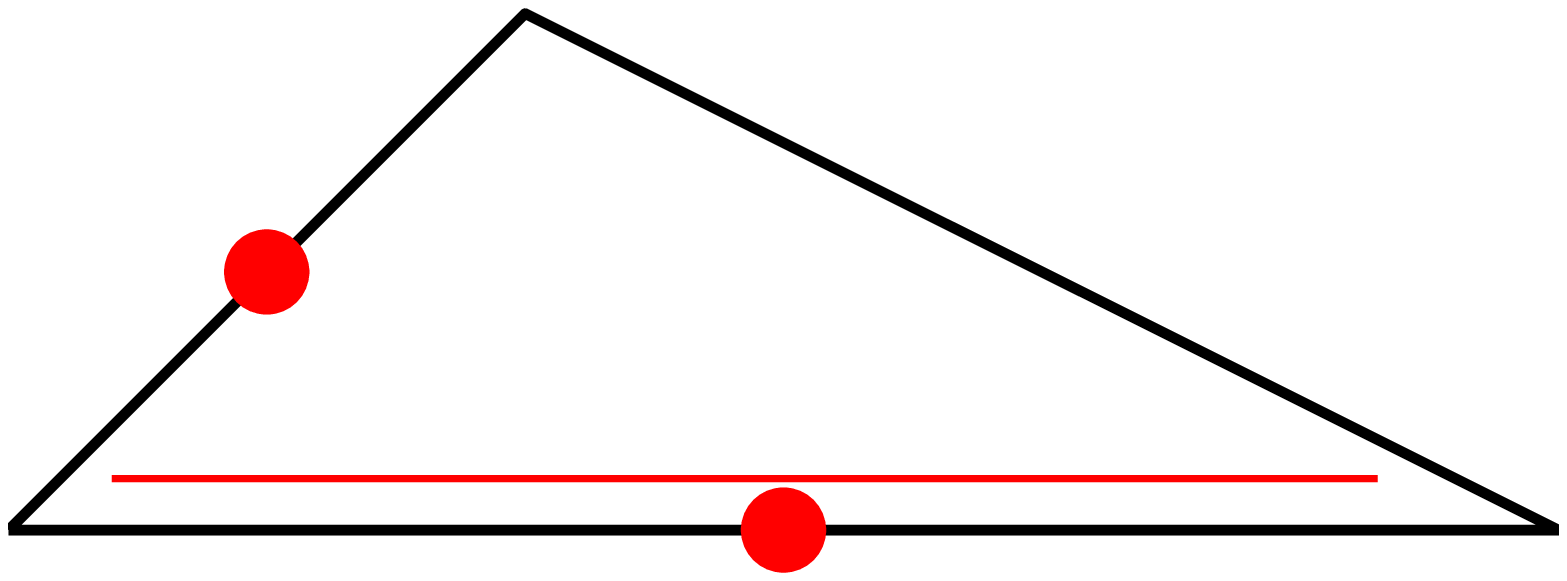} \quad
\includegraphics[width=30mm]{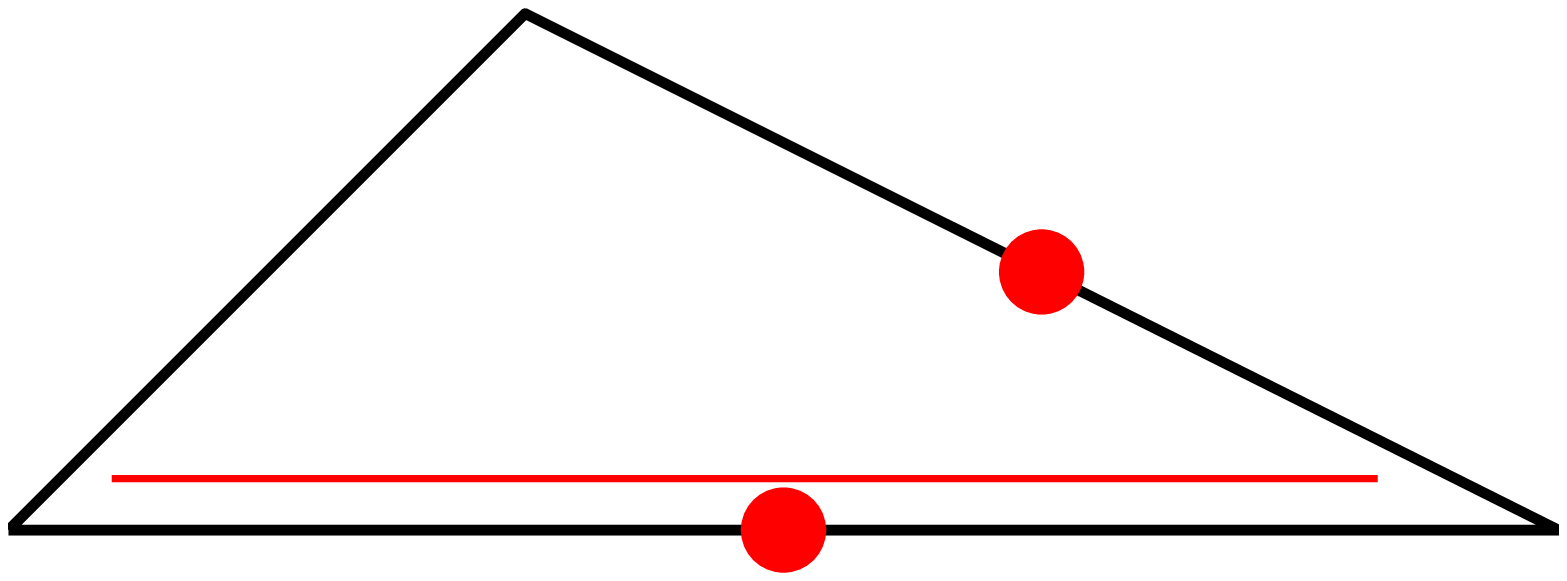} \quad
\includegraphics[width=30mm]{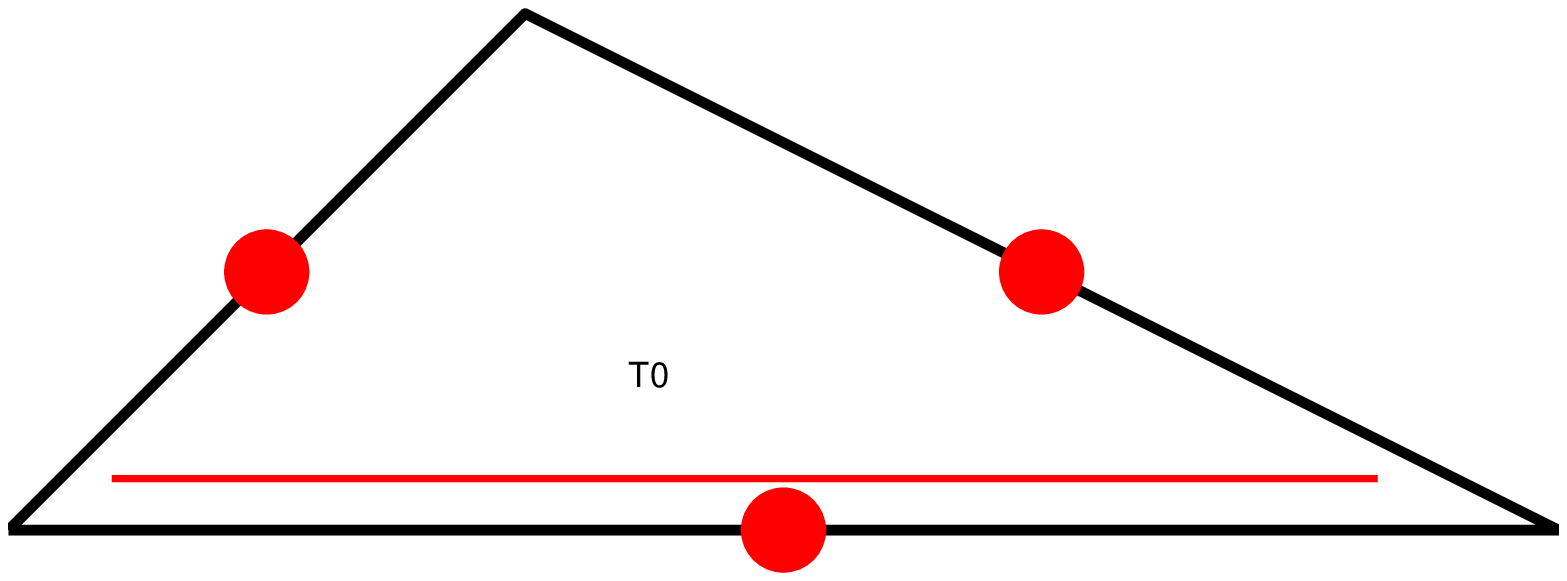} \\
\includegraphics[width=30mm]{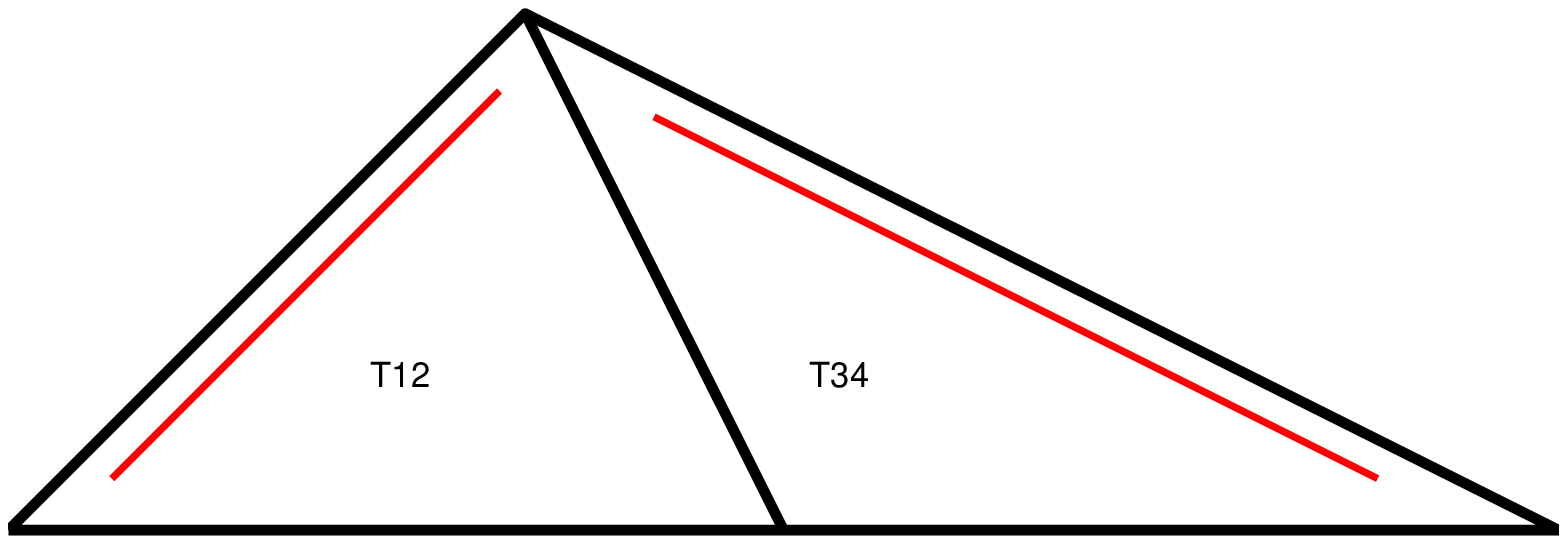} \quad
\includegraphics[width=30mm]{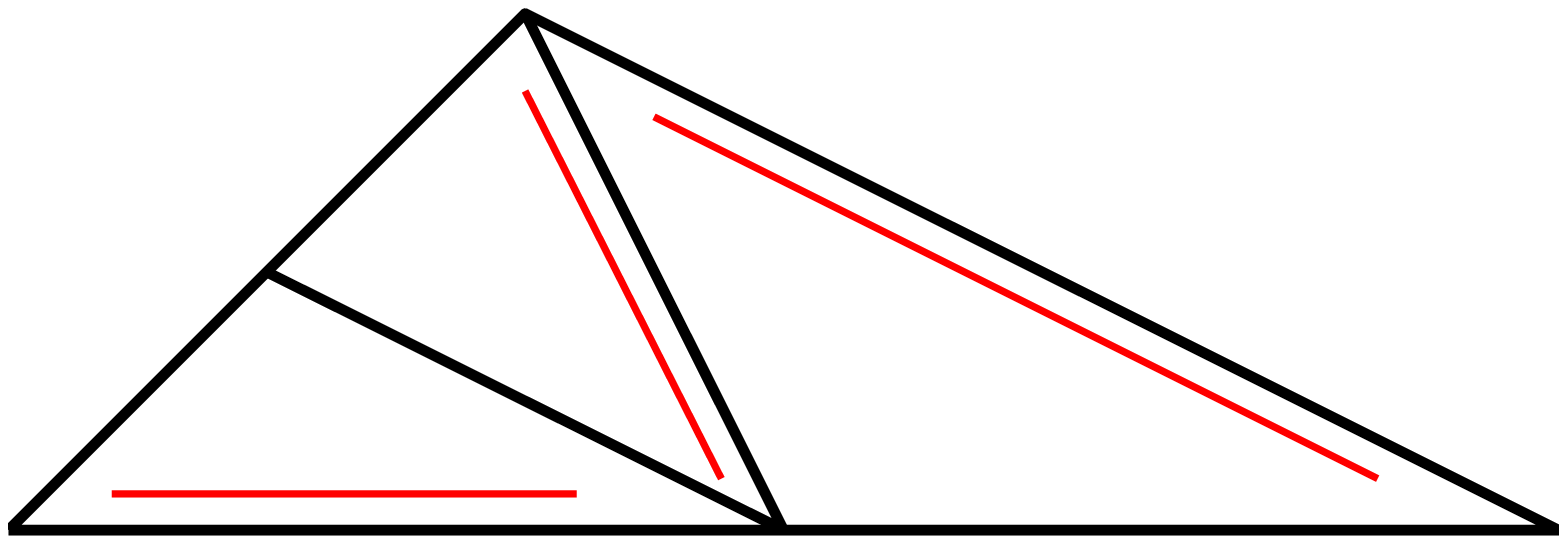}\quad
\includegraphics[width=30mm]{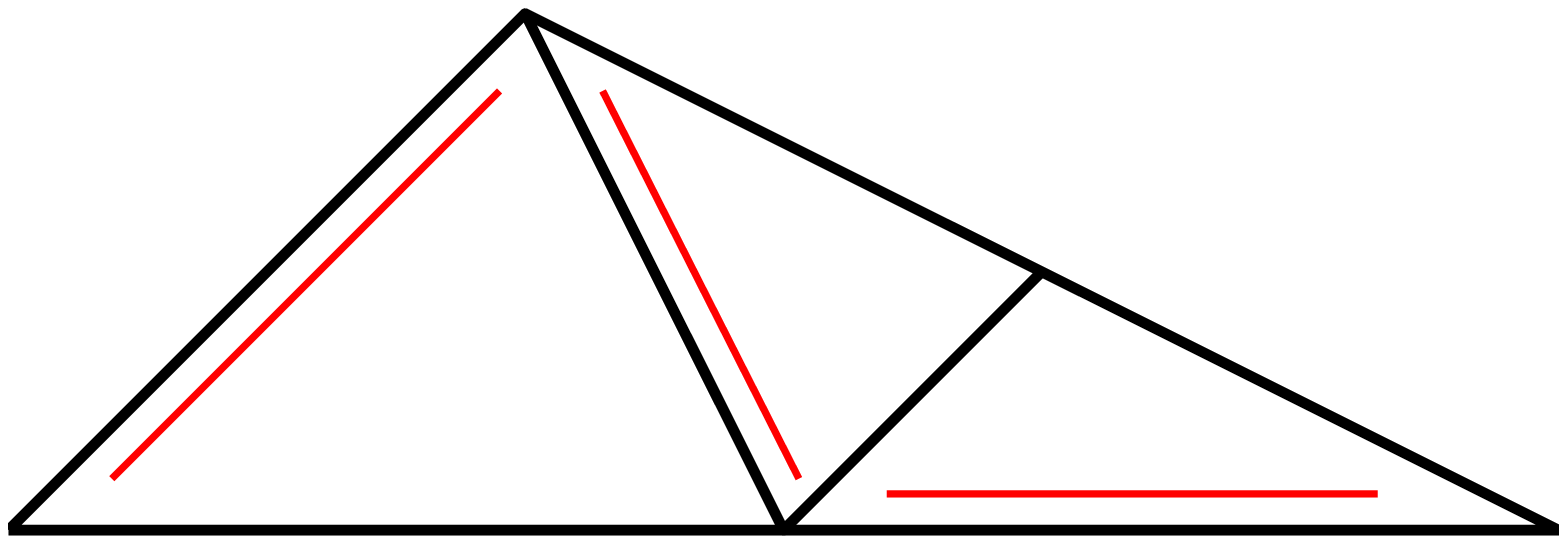}\quad
\includegraphics[width=30mm]{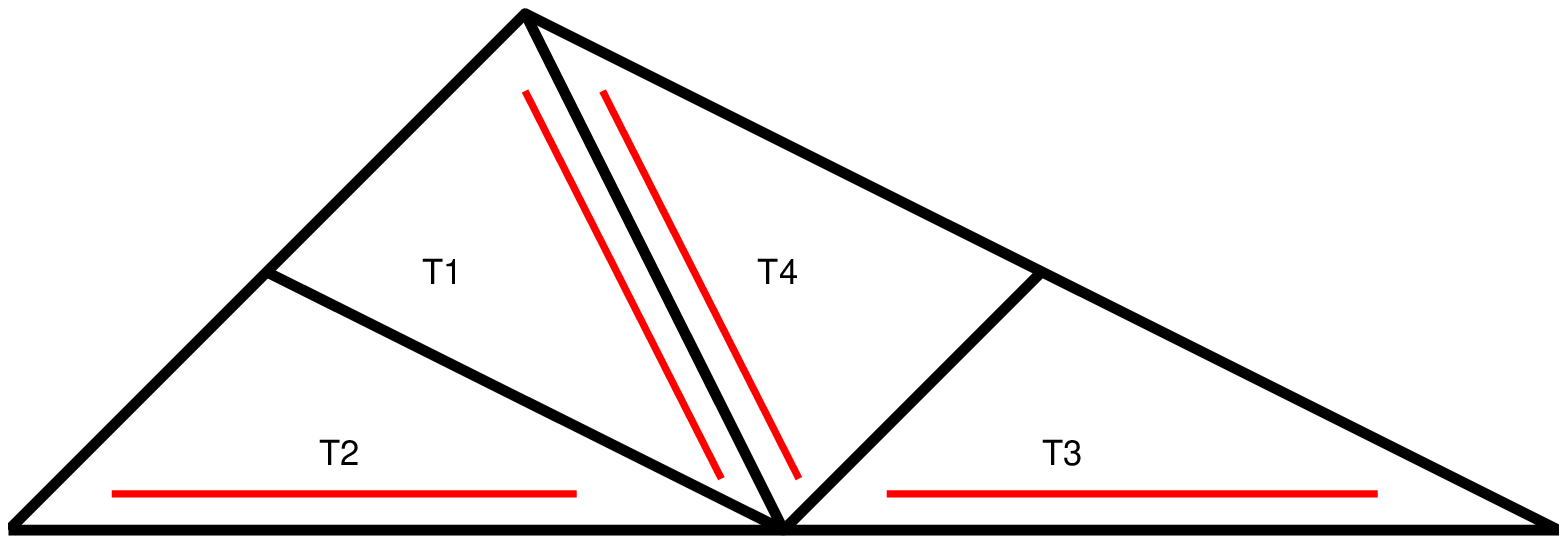}
\caption{
For each triangle $T\in\TT$, there is one fixed \emph{reference edge},
indicated by the double line (left, top). Refinement of $T$ is done by bisecting
the reference edge, where its midpoint becomes a new node. The reference
edges of the son triangles are opposite to this newest vertex (left, bottom).
To avoid hanging nodes, one proceeds as follows:
We assume that certain edges of $T$, but at least the reference edge,
are marked for refinement (top).
Using iterated newest vertex bisection, the element is then split into
2, 3, or 4 son triangles (bottom).}
\label{fig:nvb}
\end{figure}
%
\subsection{Newest vertex bisection}
\label{section:nvb}%
For the mesh-refinement, we use newest-vertex bisection. Assume that $\MM_\ell\subseteq\EE_\ell$ is a set of edges 
which have to be refined. 
The refinement rules are shown in Figure~\ref{fig:nvb}, and the
reader is also referred to~\cite[Chapter 4]{verfuerth}.
We stress a certain decay of the mesh-widths:
\begin{itemize}
\item Marked edges $E\in\MM_\ell$ are split into two edges $E',E''\in\EE_{\ell+1}$ of half length.
\item If at least one edge $E$ of an element $T\in\TT_\ell$ is marked, $T$ is refined into up to four son elements $T'\in\TT_{\ell+1}$ with area $|T|/4\le |T'|\le |T|/2$, cf.\ Figure~\ref{fig:nvb}.
\end{itemize}
Moreover, given an initial mesh $\TT_0$, newest vertex bisection only leads to at most $4\cdot\#\TT_0$ similarity classes of triangles. In particular, the generated meshes are uniformly shape regular
\begin{align}
 \sup_{\ell\in\N}\sigma(\TT_\ell) < \infty
 \quad\text{with}\quad
 \sigma(\TT_\ell) = \max_{T\in\TT_\ell}\frac{\diam(T)^2}{|T|}.
\end{align}
Furthermore, the number of different shapes of e.g.\ node patches that can occur, is finite. These observations will be necessary in the scaling arguments below.

\subsection{Weakly reliable error estimator}
\label{section:estimator}%
Let $u\in\AA$ denote the continuous solution of~\eqref{eq:obstacle} and $U_\ell\in\AA_\ell$ be the discrete solution of~\eqref{eq:obstacle_discrete}
for some fixed triangulation $\TT_\ell$. To steer the adaptive mesh-refinement, we use some residual-based error estimator
\begin{align}\label{eq:eta}
\varrho_\ell ^2
:= \sum_{E\in\EE_\ell^\Omega}\big(\eta_\ell(E)^2+\osc_\ell(E)^2\big)
+ \sum_{E \in \EE_{\ell}^\Gamma} \big(\apx_\ell(E)^2 + \osc_\ell(E)^2\big)
\end{align}
which has essentially been introduced in~\cite{cbh2007} for homogeneous Dirichlet data $g=0$.
First, $\eta_\ell(E)^2$ denotes the weighted $L^2$-norms of the normal jump
\begin{align}\label{eq:etaE}
\eta_\ell(E)^2 := h_E\,\norm{[\partial_nU_\ell]}{L^2(E)}^2
\quad\text{for }E\in\EE_\ell^\Omega
\end{align}
with $h_E=\diam(E)$ the length of $E$ and $[\cdot]$ the jump over an interior edge $E=T^+\cap T^-\in\EE_\ell^\Omega$.
Second, for interior edges, $\osc_\ell(E)^2$ denotes the data oscillations of $f$
\begin{align}\label{eq:oscE}
\osc_\ell(E)^2 := |\Omega_{\ell,E}|\,\norm{f-f_{\Omega_{\ell,E}}}{L^2(\Omega_{\ell,E})}^2
\quad\text{for }E\in\EE_\ell^\Omega
\end{align}
over the patch $\Omega_{\ell,E} = T^+\cup T^-$ associated with $E$,
where the corresponding integral mean of $f$ is denoted by $f_{\Omega_{\ell,E}}=(1/|\Omega_{\ell,E}|)\,\int_{\Omega_{\ell,E}}f\,dx$.
Third, for boundary edges $E\in\EE_\ell^\Gamma$ and $T\in\TT_\ell$ the unique element with $E\subseteq\partial T\cap\Gamma$, $\varrho_\ell$ involves the weighted element residuals
\begin{align}\label{eq:oscT}
\osc_\ell(E)^2 := |T| \|f\|^2_{L^2(T)}
\quad\text{for }E\in\EE_\ell^\Gamma.
\end{align}
Finally and in order to control the approximation of the nonhomogeneous Dirichlet data $g$ by its nodal interpolant $g_\ell$, we apply an idea from~\cite{agp2009} and use
\begin{align}\label{eq:apxE}
\apx_\ell(E)^2 := h_E\,\norm{(g - g_\ell)'}{L^2(E)}^2
\quad\text{for }E\in\EE_{\ell}^\Gamma,
\end{align}
where $(\cdot)'$ denotes the arclength derivative.

To state a reliability result for the proposed error estimator $\varrho_\ell$, we need to introduce a continuous auxiliary problem. Given the discrete Dirichlet data $g_\ell\in\SS^1(\TT_\ell|_\Gamma)$, we define
\begin{align}\label{eq:auxiliary0}
 \AA_\ell^\star := \newset{v_\ell \in \hone}{v_\ell \ge \chi \text{ a.e.\ in } \Omega, \, v_\ell|_\Gamma = g_\ell\text{ a.e.\ on }\Gamma}.
\end{align}
Applying Proposition~\ref{prop:continuous:solvability} for the data $(\chi,g_\ell, f)$, we see that the auxiliary problem
\begin{align}\label{eq:auxiliary}
 \JJ(u_\ell)= \min_{v_\ell\in\AA_\ell^\star}\JJ(v_\ell)
\end{align}
admits a unique solution $u_\ell\in\AA_\ell^\star$. The following theorem now states weak reliability of $\varrho_\ell$ in the sense that $\varrho_\ell\to0$ implies $\norm{u-U_\ell}{H^1(\Omega)}\to0$ as $\ell\to\infty$. The proof is given in Section~\ref{section:reliability} below.%

\begin{theorem}\label{thm:reliability}
With $u_\ell\in\AA_\ell^\star$ the solution of the auxiliary problem~\eqref{eq:auxiliary}, the error estimator $\varrho_\ell$ from~\eqref{eq:eta} satisfies
\begin{align}\label{eq:rel1}
\frac{1}{2}\enorm{u_\ell - U_\ell}^2 \le \JJ(U_\ell) - \JJ(u_\ell) \le \c{reliable}^2\,\widetilde \varrho_\ell^{\,2},
\end{align}
with
\begin{align}\label{eq:tilde_eta}
\widetilde\varrho_\ell^{\,2} := \sum_{E\in\EE_\ell^\Omega}\big(\eta_\ell(E)^2+\osc_\ell(E)^2\big)
+ \sum_{E\in\EE_\ell^\Gamma}\osc_\ell(E)^2,
\end{align}
as well as
\begin{align}\label{eq:reliability}
 \norm{u - U_\ell}{H^1(\Omega)}
 \le \enorm{u - u_\ell} + \sqrt{2}\c{reliable}\,\varrho_\ell,
\end{align}
where the constant $\setc{reliable}>0$ depends only on $\sigma(\TT_\ell)$ and on $\Omega$. Moreover, there holds
\begin{align}\label{eq:weakreliability}
 \varrho_\ell \xrightarrow{\ell\to\infty} 0
 \quad\Longrightarrow\quad
 \norm{u-U_\ell}{H^1(\Omega)} \xrightarrow{\ell\to\infty} 0.
\end{align}
\end{theorem}

\begin{remark}
We stress that the reliability estimate \eqref{eq:rel1} depends on the use of newest vertex bisection in the sense that only finitely many shapes of node patches and edge patches can occur. For red-green-blue refinement \cite[Chapter 4]{verfuerth}, the equivalence of edge and node oscillations is open. The entire analysis, however, applies for a coarser error estimator, where edge oscillations are bounded in terms of the element residuals $\|h_\ell f\|_{L^2(T)}$.
\end{remark}

\subsection{Convergent adaptive mesh-refining algorithm}
\label{section:adaptivity}%
We can now state our version of the adaptive algorithm in the usual form:
\begin{align*}
\boxed{\texttt{Solve}\!} \, \longmapsto \,
\boxed{\texttt{Estimate}\!} \, \longmapsto \,
\boxed{\texttt{Mark}\!} \, \longmapsto \,
\boxed{\texttt{Refine}\!}
\end{align*}
Throughout, we assume that the Galerkin solution $U_\ell\in\SS^1(\TT_\ell)$ is computed exactly. For marking, we use the strategy proposed by D\"orfler~\cite{d1996}.

\begin{algorithm}\label{algorithm}
Fix an adaptivity parameter $0<\theta<1$, let $\TT_\ell$
with $\ell=0$ be the initial triangulation, and fix a reference edge for each element $T\in\TT_0$. For each $\ell=0,1,2,\dots$ do:
\begin{itemize}
\item[(i)] Compute discrete solution $U_\ell\in \AA_\ell$.
\item[(ii)] Compute $\eta_\ell(E)^2$, $\osc_\ell(E)^2$, and $\apx_\ell(E)^2$
for all $E\in\EE_\ell$.
\item[(iii)] Determine (minimal) set $\MM_\ell\subseteq\EE_\ell$
which satisfies
\begin{align}\label{eq:doerfler}
\theta\,\varrho_\ell^2
\le\sum_{E\in\EE_\ell^\Omega\cap\MM_\ell}\big(\eta_\ell(E)^2+\osc_\ell(E)^2\big)
+ \sum_{E \in \EE_{\ell}^\Gamma\cap\MM_\ell} \big(\apx_\ell(E)^2+\osc_\ell(E)^2\big).
\end{align}
\item[(iv)] Mark all edges $E\in\MM_\ell$ and obtain
new mesh $\TT_{\ell+1}$ by newest vertex bisection.
\item[(v)] Increase counter $\ell\mapsto\ell+1$ and iterate.\qed
\end{itemize}
\end{algorithm}

The following theorem states our main convergence results.

\begin{theorem}\label{thm:convergence}
Algorithm~\ref{algorithm} guarantees the existence of
constants $0<\kappa,\gamma<1$ and $\lambda>0$ and a sequence $\alpha_\ell\ge0$ with $\lim\limits_{\ell\to\infty}\alpha_\ell=0$ such that the combined error quantity
\begin{align}\label{eq:combined_error}
 \Delta_\ell := \JJ(U_\ell) - \JJ(u_\ell) + \gamma \varrho_\ell^2 + \lambda\,\apx_\ell^2\ge0
\end{align}
satisfies contraction up to the zero sequence $\alpha_\ell$, i.e.\
\begin{align}\label{eq1:conv}
 \Delta_{\ell+1} \le \kappa\,\Delta_\ell + \alpha_\ell
 \quad\text{for all }\ell\in\N.
\end{align}
In particular, this implies
\begin{align}\label{eq2:conv}
 0 = \lim_{\ell\to\infty}\Delta_\ell
 = \lim_{\ell\to\infty}\varrho_\ell
 = \lim_{\ell\to\infty}\norm{u-U_\ell}{H^1(\Omega)}
\end{align}
as well as convergence of the energies
\begin{align}
 \JJ(u) = \lim_{\ell\to\infty}\JJ(u_\ell)
 = \lim_{\ell\to\infty}\JJ(U_\ell).
\end{align}
\end{theorem}
%

\section{Proof of Theorem~\ref{thm:reliability} (Weak Reliability of Error Estimator)}
\label{section:reliability}

\subsection{Stability of continuous problem}
\label{section:stability}
For the finite element discretization, we have to replace the continuous Dirichlet data $g\in H^{1/2}(\Gamma)$ by appropriate discrete functions $g_\ell$. To make this procedure feasible, we have to prove that the solution of the obstacle problem~\eqref{eq:obstacle} depends continuously on the given data.

In the following, let $\HH$ be a Hilbert space with scalar product
$\edual\cdot\cdot$ and $\KK, \KKell \subset \HH$ be closed, convex, and non-empty
subsets of $\HH$. We say, that $\KKell$
converges to $\KK$ \emph{in the sense of Mosco} if and only if
\begin{align}\label{eq:mosco1}
 \forall v\in \KK \, \exists v_\ell\in \KKell:
 \quad v_\ell\to v\text{ strongly in }\HH
 \text{ as }\ell\to\infty
\end{align}
and
\begin{align}\label{eq:mosco2}
 \forall v\in\HH \, \forall v_\ell\in \KKell:
 \quad\big(v_\ell \rightharpoonup v \text{ weakly in } \HH
 \text{ as }\ell\to\infty\quad
 \Rightarrow\quad v\in \KK\big).
\end{align}
We can now state an abstract stability result of~{Mosco}:

\begin{lemma}[{\cite[Theorem A]{mosco}}]\label{lem:mosco}
Assume the sets $\KK,\KKell$ satisfy~\eqref{eq:mosco1}--\eqref{eq:mosco2}.
Let $L,L_\ell\in\HH^*$ be functionals and assume $L_\ell\to L$ in
$\HH^*$ as $\ell\to\infty$. Finally, let $u\in \KK$ and $u_\ell\in \KKell$ be
the unique solutions of~\eqref{eq:minimization} with respect to the data
$(\KK,L)$ and $(\KKell,L_\ell)$, respectively. Then there holds strong
convergence
\begin{align}\label{eq:abstract_problem:limit}
 u_\ell \to u\text{ in }\HH\text{ as }\ell\to\infty,
\end{align}
i.e.\ the solution of the variational problem~\eqref{eq:minimization} depends
continuously on the given data.\qed
\end{lemma}

\begin{remark}
We remark that we have stated~\cite[Theorem A]{mosco} only in a simplified form.
In general, the preceding lemma of Mosco includes the approximation of the bilinear form
$\edual\cdot\cdot$ as well, and it also holds for nonlinear variational
inequalities, where the underlying operators $A_\ell,A:\HH\to\HH^*$,
e.g.\ $Av:=\edual{v}{\cdot}$, are
assumed to be Lipschitz continuous and strictly monotone with constants
independent of $\ell$.
\end{remark}

We are now in the position to show that the
solution $u\in\AA$ of the obstacle problem~\eqref{eq:obstacle} continuously depends on the boundary data $g\in H^{1/2}(\Gamma)$.

\begin{proposition}\label{thm:stability}
Recall that $u_\ell\in\AA_\ell^\star$ denotes the solution of the auxiliary problem~\eqref{eq:auxiliary}.
Provided convergence $g_\ell\to g$ in $H^{1/2}(\Gamma)$ of the Dirichlet data, there also holds convergence $u_\ell\to u$ in $H^1(\Omega)$ of the (continuous) solutions as $\ell\to\infty$.
\end{proposition}

\begin{proof}
As in the proof of Proposition~\ref{prop:continuous:solvability},
we define the sets
\begin{align*}
\KK = \newset{\widetilde v \in H^1_0(\Omega)}{\widetilde v \ge \chi - \LL g}
\quad\text{and}\quad
\KKell = \newset{\widetilde v_\ell \in H^1_0(\Omega)}{\widetilde v_\ell \ge \chi - \LL g_\ell}.
\end{align*}
We stress that both sets are convex, closed, and non-empty subsets of $H^1_0(\Omega)$. As usual, let $H^{-1}(\Omega) := H^1_0(\Omega)^*$ denote the dual space of $H^1_0(\Omega)$.
With respect to the abstract notation,
we have $L,L_\ell\in H^{-1}(\Omega)$ with
\begin{align*}
 \dual{L}{\widetilde v} = (f,\widetilde v) - \edual{\LL g}{\widetilde v}
 \quad\text{and}\quad
 \dual{L_\ell}{\widetilde v} = (f,\widetilde v) - \edual{\LL g_\ell}{\widetilde v}.
\end{align*}
Note that this implies
\begin{align*}
 \norm{L-L_\ell}{H^{-1}(\Omega)} \le \norm{\LL g-\LL g_\ell}{H^1(\Omega)}
 \lesssim \norm{g-g_\ell}{H^{1/2}(\Gamma)}\xrightarrow{\ell\to\infty}0.
\end{align*}
Next, we prove that $\KKell\to\KK$ in the sense of Mosco~\eqref{eq:mosco1}--\eqref{eq:mosco2}.

To verify~\eqref{eq:mosco1}, let $\widetilde v\in \KK$ be arbitrary. Define
\begin{align*}
 v := \widetilde v + \LL g \in \AA
\end{align*}
and let $w_\ell \in H^1(\Omega)$ be the unique solution of the (linear) auxiliary problem
\begin{align*}
 \Delta w_\ell = 0 \text{ in } \Omega
 \quad\text{with Dirichlet conditions}\quad
 w_\ell = g_\ell - g \text{ on }\Gamma.
\end{align*}
In the next step, we consider
\begin{align*}
 \overline v_\ell := v + w_\ell \in H^1(\Omega).
\end{align*}
Then, there holds $\overline v_\ell|_\Gamma = g_\ell$ as well as
\begin{align*}
 \norm{v - \overline v_\ell}{H^1(\Omega)}
 = \norm{w_\ell}{H^1(\Omega)}
 \lesssim \norm{g_\ell - g}{H^{1/2}(\Gamma)}
 \xrightarrow{\ell\to\infty}0.
\end{align*}
Since the maximum of $H^1$-functions belongs to $H^1$, we obtain
\begin{align*}
 v_\ell := \max\{\chi,\overline v_\ell\}\in H^1(\Omega).
\end{align*}
By definition, there holds $v_\ell \ge \chi$ almost everywhere in $\Omega$ as well as $v_\ell|_\Gamma = g_\ell$ since $\overline v_\ell|_\Gamma = g_\ell \ge \chi|_\Gamma$ almost everywhere on $\Gamma$. Altogether, this proves $v_\ell\in\AA_\ell$. Moreover, since the function
$\max\{\chi,\cdot\}:H^1(\Omega)\to H^1(\Omega)$ is well-defined and continuous (see e.g.~\cite[Chapter II, Corollary 2.1]{g08}), there holds
\begin{align*}
 v = \max\{\chi,v\} = \lim_{\ell\to\infty} \max\{\chi,\overline v_\ell\} = \lim_{\ell\to\infty} v_\ell
 \text{ in }H^1(\Omega).
\end{align*}
Finally, we observe
\begin{align*}
 \widetilde v_\ell := v_\ell - \LL g_\ell
 \xrightarrow{\ell\to\infty} v - \LL g = \widetilde v
 \text{ in }H^1(\Omega)
\end{align*}
and $\widetilde v_\ell\in \KKell$.

To verify~\eqref{eq:mosco2}, let $v_\ell \in \KKell$ and assume that the weak limit $v_\ell\rightharpoonup v \in H^1_0(\Omega)$ exists as $\ell\to\infty$. According to the Rellich compactness theorem, there is a subsequence $(v_{\ell_k})$ which converges strongly to $v$ in $L^2(\Omega)$.
According to the Weyl theorem, we may thus extract a further subsequence $(v_{\ell_{k_j}})$ which converges to $v$ pointwise almost everywhere in $\Omega$. Moreover, by continuity of the lifting operator, there holds $\LL g_{\ell_{k_j}} \to \LL g$ in $H^1(\Omega)$ as $j\to\infty$. Since this implies $L^2$-convergence and again according to the Weyl theorem, we may choose a subsequence $(\LL g_{\ell_{k_{j_i}}})$ which converges to $\LL g$ pointwise almost everywhere in $\Omega$. Altogether, the estimate
\begin{align*}
 v_{\ell_{k_{j_i}}} \ge \chi-\LL g_{\ell_{k_{j_i}}}
 \quad\text{a.e.\ in }\Omega
\end{align*}
and monotonicity of the pointwise limit imply
\begin{align*}
 v \ge \chi - \LL g
 \quad\text{a.e.\ in }\Omega,
\end{align*}
whence $v\in \KK$.

Now, we may apply Lemma~\ref{lem:mosco} to see that the unique solutions $\widetilde u\in \KK$ and $\widetilde u_\ell\in \KKell$ of the minimization problems~\eqref{eq:minimization} with respect to the data $(\KK,L)$ and $(\KKell,L_\ell)$, respectively, guarantee
\begin{align*}
 \norm{\widetilde u - \widetilde u_\ell}{H^1(\Omega)} \xrightarrow{\ell\to\infty}0.
\end{align*}
With Lemma~\ref{lem:zeroboundary}, there holds $u = \widetilde u + \LL g$ as well as $u_\ell = \widetilde u_\ell + \LL g_\ell$. Altogether, we thus obtain
\begin{align*}
 \norm{u - u_\ell}{H^1(\Omega)}
 \lesssim \norm{\widetilde u - \widetilde u_\ell}{H^1(\Omega)}
 + \norm{g-g_\ell}{H^{1/2}(\Gamma)}
 \xrightarrow{\ell\to\infty}0
\end{align*}
and conclude the proof.
\end{proof}

\subsection{Some elementary preliminaries}
The next result is an elementary lemma for variational inequalities, see e.g.\ \cite[Proposition~2]{pp2010}.

\begin{lemma}\label{lem:rel_estimates}
Suppose that $\AA$ is a convex subset of $H^1(\Omega)$ and $f\in L^2(\Omega)$. Let $U\in\AA$ be the solution of the variational inequality
\begin{align}
\edual{U}{V-U} \ge (f,V-U)\quad\text{for all }V\in\AA.
\end{align}
Then, for all $W \in \AA$, there holds
\begin{align}
 \frac12\,\enorm{U-W}^2 \le \JJ(W) - \JJ(U),
\end{align}
i.e.\ the difference in the $H^1$-seminorm is controlled in terms of the energy.\qed
\end{lemma}

\subsection{Proof of Theorem~\ref{thm:reliability}}
We start with estimate \eqref{eq:rel1}, where the lower estimate clearly holds due to Lemma~\ref{lem:rel_estimates}. We thus only need to show the upper estimate. To that end,
we define $\sigma_\ell \in H^{-1}(\Omega)$ by
\begin{align*}
\dual{\sigma_\ell}{v} := (f,v) - \edual{u_\ell}{v} \quad \text{ for all } v \in H^1_0(\Omega).
\end{align*}
Since $u_\ell|_\Gamma = U_\ell|_\Gamma$, we may argue as in the proof of \cite[Theorem 1]{cbh2007} to see
\begin{align*}
\enorm{u_\ell - U_\ell}^2 + \dual{\sigma_\ell}{u_\ell - U_\ell} \lesssim \widetilde\varrho_\ell^{\,2},
\end{align*}
where the hidden constant depends only on the shape regularity constant $\sigma(\TT_\ell)$.
A direct calculation finally shows
\begin{align*}
\JJ(U_\ell) - \JJ(u_\ell) &= \frac12\edual{U_\ell}{U_\ell} - (f,U_\ell) - \Big(\frac12\edual{u_\ell}{u_\ell} - (f,u_\ell)\Big)\\
&\quad + \big(\edual{u_\ell}{u_\ell - U_\ell} - (f,u_\ell - U_\ell) + \dual{\sigma_\ell}{u_\ell - U_\ell}\big)\\
&= \frac12\enorm{u_\ell - U_\ell}^2 + \dual{\sigma_\ell}{u_\ell - U_\ell},
\end{align*}
whence $\JJ(U_\ell) - \JJ(u_\ell) \lesssim \widetilde \varrho_\ell^{\,2}$.

Next we prove \eqref{eq:reliability}.
According to the Rellich compactness theorem and the triangle inequality, there holds
\begin{align*}
 \norm{u - U_\ell}{H^1(\Omega)}
 &\simeq \enorm{u - U_\ell} + \norm{g- g_\ell}{H^{1/2}(\Gamma)}\\
 &\le \enorm{u - u_\ell} + \enorm{u_\ell - U_\ell}
  + \norm{g- g_\ell}{H^{1/2}(\Gamma)}.
\end{align*}
From \eqref{eq:rel1}, we infer
\begin{align*}
\enorm{u_\ell - U_\ell}^2 \lesssim \widetilde \varrho_\ell^{\,2} \le \varrho_\ell^2.
\end{align*}
From \cite[Lemma 3.3]{agp2009}, we additionally infer that
\begin{align*}
 \norm{g-g_\ell}{H^{1/2}(\Gamma)}^2
 \lesssim \sum_{E\in\EE_{\ell}^\Gamma}\apx_\ell(E)^2
 \le \varrho_\ell^2,
\end{align*}
and the constant depends only on $\Gamma$ and the local mesh ratio of $\EE_{\ell}^\Gamma=\TT_\ell|_\Gamma$, whence on $\Omega$ and $\sigma(\TT_\ell)$.
The combination of the last three estimates yields~\eqref{eq:reliability}.
Moreover, according to the last estimate, the convergence $\varrho_\ell\to0$ implies $g_\ell\to g$ in $H^{1/2}(\Gamma)$ as $\ell\to\infty$. According to the stability result from Proposition~\ref{thm:stability}, this yields convergence $u_\ell \to u$ in $H^1(\Omega)$, and we also conclude the proof of~\eqref{eq:weakreliability}.
\qed

\section{Proof of Theorem~\ref{thm:convergence} (Convergence of AFEM)}
\label{section:convergence}

The idea of our convergence proof is roughly sketched as follows: D\"orfler marking~\eqref{eq:doerfler} 
yields that $\varrho_\ell$ is contractive up to $\enorm{U_{\ell+1}-U_\ell}$, see Section~\ref{section:reduction}. Since there seems to be no estimate for this term to be available (recall nonconformity $\AA_\ell\not\subseteq\AA_{\ell+1}$), we introduce a discrete auxiliary problem with solution $U_{\ell+1,\ell}$ which allows to control
\begin{align*}
 \enorm{U_{\ell+1}-U_\ell}
 \le \enorm{U_{\ell+1}-U_{\ell+1,\ell}} + \enorm{U_{\ell+1,\ell}-U_\ell}
\end{align*}
in terms of the energy $\JJ(U_\ell)-\JJ(U_{\ell+1,\ell})$, see Section~\ref{section:energy}.
In Section~\ref{section:convergence:proof}, we finally follow the concept of~\cite{ckns} to prove Theorem~\ref{thm:convergence}.

\subsection{Estimator reduction}
\label{section:reduction}%
The following contraction estimate is called \emph{estimator reduction} in~\cite{afp2009,ckns}.

\begin{proposition}\label{proposition:reduction}
Suppose that the set $\MM_\ell\subseteq\EE_\ell$
satisfies~\eqref{eq:doerfler} and that marked edges 
are refined as stated in Section~\ref{section:nvb}. Then, there holds
\begin{align}\label{eq:reduction}
\varrho_{\ell+1}^2
\le q\,\varrho_\ell^2 + \c{reduction}\enorm{U_{\ell+1}-U_\ell}^2
\end{align}
with some contraction constant $q\in(0,1)$ which depends only on
$\theta\in(0,1)$. The constant $\setc{reduction}>0$ additionally depends
only on the initial mesh $\TT_0$.
\end{proposition}

Since the proof of Proposition~\ref{proposition:reduction} follows along the lines of the proof of~\cite[Proposition~3]{pp2010}, we only sketch it for brevity.

\begin{proof}[Sketch of proof]
First, the Young inequality proves for arbitrary $\delta>0$
\begin{align}\label{eq:reduction:young}
 \begin{split}
 \varrho_{\ell+1}^2
 &\le
 \sum_{E'\in\EE_{\ell+1}}\osc_{\ell+1}(E')^2
 + \sum_{E' \in \EE_{\ell+1}^\Gamma} \apx_{\ell+1}(E')^2
 \\&+
 (1+\delta)\sum_{E'\in\EE_{\ell+1}^\Omega}
 h_{E'}\,\norm{[\partial_nU_{\ell}]}{L^2(E')}^2
 + (1+\delta^{-1})\sum_{E'\in\EE_{\ell+1}^\Omega}
 h_{E'}\,\norm{[\partial_n(U_{\ell+1}-U_{\ell})]}{L^2(E')}^2.
 \end{split}
\end{align}
A scaling argument allows to estimate the last term by
\begin{align}
 \sum_{E'\in\EE_{\ell+1}^\Omega}
 h_{E'}\,\norm{[\partial_n(U_{\ell+1}-U_{\ell})]}{L^2(E')}^2
 \le C\,\enorm{U_{\ell+1}-U_\ell}^2,
\end{align}
and the constant $C>0$ depends only on $\sigma(\TT_\ell)$. Next, one investigates the reduction of the other three terms if the mesh is refined locally:
According to~\cite[Lemma~5]{pp2010}, it holds that
\begin{align}\label{eq:reductionETA}
\sum_{E'\in\EE_{\ell+1}^\Omega}
h_{E'}\,\norm{[\partial_nU_{\ell}]}{L^2(E')}^2
\le \sum_{E\in\EE_\ell^\Omega}
\eta_\ell(E)^2
- \frac12\,\sum_{E\in\EE_\ell^\Omega\cap\MM_\ell}\eta_\ell(E)^2.
\end{align}
Next, \cite[Lemma~6]{pp2010} resp.\ \cite[Lemma 3.9.6]{page2010} yields
\begin{align}\label{eq:reductionOSCE}
\sum_{E'\in\EE_{\ell+1}}\osc_{\ell+1}(E')^2
\le \sum_{E\in\EE_\ell}\osc_\ell(E)^2
- \frac14\,\sum_{E\in\EE_\ell\cap\MM_\ell}\osc_\ell(E)^2.
\end{align}
Moreover, it is part of the proof of~\cite[Theorem~4.2]{agp2009} that
\begin{align}\label{eq:reductionAPXE}
\sum_{E' \in \EE_{\ell+1}^\Gamma} \apx_{\ell+1}(E')^2 \le \sum_{E \in \EE_{\ell}^\Gamma}\apx_\ell(E)^2 - \frac{1}{2}\sum_{E \in \EE_{\ell}^\Gamma\cap \MM_\ell}\apx_\ell(E)^2.
\end{align}
Combining the estimates~\eqref{eq:reduction:young}--\eqref{eq:reductionAPXE}
and using the D\"orfler marking~\eqref{eq:doerfler}, we easily obtain
\begin{align*}
 \varrho_{\ell+1}^2
 \le(1+\delta)(1-\theta/4)\,\varrho_\ell^2
 + (1+\delta^{-1})C\,\enorm{U_{\ell+1}-U_\ell}^2.
\end{align*}
Finally, we may choose $\delta>0$ sufficiently small to guarantee $q:=(1+\delta)(1-\theta/4)<1$ to end up with~\eqref{eq:reduction}.
\end{proof}

\begin{remark}
In the conforming case $\AA_\ell\subseteq\AA_{\ell+1}$, it is easily seen that the sequence $U_\ell$ of discrete solutions tends to some limit $U_\infty$ which is the unique Galerkin solution with respect to the closure of $\bigcup_{\ell=0}^\infty\AA_\ell$, see~\cite[Lemma~3.3.8]{page2010}. Therefore, $\lim_\ell\enorm{U_{\ell+1}-U_\ell}=0$, and elementary calculus predicts that~\eqref{eq:reduction} already implies $\lim_\ell\varrho_\ell=0$, cf.\ e.g.\ \cite[Proposition 1.2]{afp2009}. Then, Theorem~\ref{thm:reliability} would conclude that $\lim_\ell\norm{u-U_\ell}{H^1(\Omega)}=0$, i.e.\ $u=U_\infty$.
---
Now that $\AA_\ell\not\subseteq\AA_{\ell+1}$, we did neither succeed to prove that the a~priori limit $U_\infty$ exists nor the much weaker claim that $\lim_\ell\enorm{U_{\ell+1}-U_\ell}=0$.
\end{remark}

\subsection{Some a~priori convergence results}
\label{section:apriori}%
In the adaptive algorithm, the mesh $\TT_{\ell+1}$ is obtained
by local refinement of elements in $\TT_\ell$. Consequently, the discrete
spaces $\SS^1(\TT_\ell)$ are nested, i.e.,
\begin{align}\label{eq:nestedness}
 \SS^1(\TT_{\ell}) \subseteq \SS^1(\TT_{\ell+1})
 \quad\text{for all }\ell\in\N
\end{align}
and the analogous inclusion also holds for the spaces
$\SS^1(\TT_{\ell}|_\Gamma) \subseteq \SS^1(\TT_{\ell+1}|_\Gamma)$ on the
boundary. The following lemma is part of the proof
of~\cite[Proposition~10]{afp2010}.

\begin{lemma}\label{lemma:apriori:boundary}
The nodal interpolants $g_\ell\in\SS^1(\TT_{\ell}|_\Gamma)$ of some
Dirichlet data $g\in H^1(\Gamma)$ converge to some a~priori limit in
$H^1(\Gamma)$, i.e.\ there holds
\begin{align}\label{eq:apriori:boundary}
 \norm{g_\infty - g_\ell}{H^1(\Gamma)}
 \xrightarrow{\ell\to\infty}0
\end{align}
for a certain element $g_\infty\in H^1(\Gamma)$.\qed
\end{lemma}

In the following, we will only use the convergence of $g_\ell$ to $g_\infty$ in $H^{1/2}(\Gamma)$ as well as the uniform boundedness $\sup\limits_{\ell\in\N}\norm{g_\ell}{H^{1/2}(\Gamma)} < \infty$ which is an immediate consequence.

\begin{proposition}\label{cor:apriori:obstacle}
The sequence $u_\ell\in\AA_\ell^\star$ of solutions of the continuous auxiliary problem~\eqref{eq:auxiliary} converges to some a~priori limit $u_\infty\in H^1(\Omega)$, i.e.
\begin{align}
 \norm{u_\infty - u_\ell}{H^1(\Omega)}
 \xrightarrow{\ell\to\infty}0
\end{align}
for some function $u_\infty\in H^1(\Omega)$. In particular, there holds
\begin{align}
 \JJ(u_\infty) = \lim_{\ell\to\infty}\JJ(u_\ell).
\end{align}
\end{proposition}

\begin{proof}
Since convergence in $H^{1/2}(\Gamma)$ implies convergence in $L^2(\Gamma)$, the Weyl theorem applies and proves that a subsequence $(g_{\ell_k})$ converges to $g_\infty$ pointwise almost everywhere on $\Gamma$. Therefore, the limit function $g_\infty\in H^{1/2}(\Gamma)$ satisfies $g_\infty\ge\chi$ almost everywhere on $\Gamma$. In particular, the obstacle problem~\eqref{eq:obstacle} with $g=g_\infty$ has also a unique solution $u_\infty\in H^1(\Gamma)$. Applying the stability result from Proposition~\ref{thm:stability}, we obtain convergence of $u_\ell$ to $u_\infty$ in $H^1(\Omega)$ as $\ell\to\infty$.
\end{proof}

\subsection{Discrete energy estimates}
\label{section:energy}%
The main difficulty which we suffered in our analysis is that it is not clear that $\enorm{U_{\ell+1}-U_\ell}$ resp.\ $\JJ(U_{\ell+1})-\JJ(U_\ell)$ tend to zero as $\ell\to\infty$. We circumvent this question by introducing a discrete auxiliary problem: \emph{Find $U_{\ell+1,\ell}\in\AA_{\ell+1,\ell}$ such that}
\begin{align}\label{eq:auxiliary:discrete}
\JJ(U_{\ell+1,\ell}) = \min_{V_{\ell+1}\in \AA_{\ell+1,\ell}}\JJ(V_{\ell+1}).
\end{align}
where the admissible set reads
\begin{align}\label{eq:auxiliary:discrete2}
 \AA_{\ell+1,\ell} := \set{V_{\ell+1} \in \SS^1(\TT_{\ell+1})}{V_{\ell+1} \ge \chi \text{ in }\Omega, \, V_{\ell+1} = g_\ell \text{ on }\Gamma}.
\end{align}
Applying Proposition~\ref{prop:discrete:solvability} to the data $(\AA_{\ell+1,\ell},g_\ell)$ instead of $(\AA_\ell,g_\ell)$, we see that the auxiliary problem~\eqref{eq:auxiliary:discrete} admits a unique solution
$U_{\ell+1,\ell}\in\AA_{\ell+1,\ell}$.

Our first lemma is a key ingredient of the upcoming proofs. It states that one may change the boundary data of a discrete function and control the influence within $\Omega$.

\begin{lemma}\label{lem:stability_boundary}
Let $W_{\ell+1} \in \SS^1(\TT_{\ell+1})$ with $W_{\ell+1}|_\Gamma = g_{\ell+1}$. Define $W_{\ell+1}^\ell \in \SS^1(\TT_{\ell+1})$ by
\begin{align}\label{eq:boundary_changed}
 W_{\ell+1}^\ell(z) = \begin{cases}
 W_{\ell+1}(z) & \text{ for } z\in \NN_{\ell+1}\backslash \Gamma,\\
 g_\ell(z) & \text{ for }z \in \NN_{\ell+1}\cap \Gamma.\end{cases}
\end{align}
Then, \dpr{with the local mesh-width $h_\ell \in L^\infty(\Gamma)$ defined by $h_\ell|_E = h_E$,}
there holds
\begin{align}\label{eq:stability_boundary}
 \norm{W_{\ell+1} - W_{\ell+1}^\ell}{H^1(\Omega)}
 \le \c{boundary2}\,\norm{h_\ell^{1/2}(g_{\ell+1} - g_\ell)'}{L^2(\Gamma)}
 \le \c{boundary}\,\norm{g_{\ell+1} - g_\ell}{H^{1/2}(\Gamma)},
\end{align}
where $\setc{boundary2},\setc{boundary}>0$ depend only on $\TT_0$ and on $\Omega$.
\end{lemma}

\def\Tref{T_{\rm ref}}
\def\zref{z^{\rm ref}}
\begin{proof}
By definition, we have $W_{\ell+1}^\ell \in \SS^1(\TT_{\ell+1})$ with $W_{\ell+1}^\ell|_\Gamma = g_\ell$. Moreover, there holds
$W_{\ell+1}^\ell |_T = W_{\ell+1}|_T$ for all $T \in \TT_{\ell+1}\setminus\TT_{\ell+1}^\Gamma$.
Let $T =\conv\{z_1,z_2,z_3\} \in \TT_{\ell+1}^\Gamma$ and $\Phi_T : T_{\text{ref}} \rightarrow T$ be the affine diffeomorphism with reference element $T_{\text{ref}} = \conv\{(0,0),(0,1),(1,0)\}$.

The transformation formula and norm equivalence on $\PP^1(\Tref)$ prove
\begin{align*}
 \norm{W_{\ell+1}-W_{\ell+1}^\ell}{L^2(T)}
 &=|T|^{1/2} \, \norm{(W_{\ell+1}-W_{\ell+1}^\ell)\circ \Phi_{T}}{L^2(\Tref)}\\
 &\lesssim |T|^{1/2} \sum_{j=1}^3 |(W_{\ell+1}-W_{\ell+1}^\ell)\circ \Phi_{T}(\zref_j)|\\
 &\leq|\Omega|^{1/2} \sum_{j=1}^3 |(W_{\ell+1}-W_{\ell+1}^\ell)(z_j)|
\end{align*}
as well as
\begin{align*}
 \norm{\nabla(W_{\ell+1}-W_{\ell+1}^\ell)}{L^2(T)}
 &\lesssim \sigma(\TT_{\ell}) \, \norm{\nabla\big((W_{\ell+1}-W_{\ell+1}^\ell)\circ \Phi_{T}\big)}{L^2(\Tref)}\\
 &\lesssim \sigma(\TT_{\ell}) \sum_{j=1}^3 |(W_{\ell+1}-W_{\ell+1}^\ell)(z_j)|.
\end{align*}
Now, we consider the nodes $z\in \{z_1,z_2,z_3\}$ of the triangle $T$:
For $z\in\NN_\ell\cap\Gamma$ holds $g_\ell(z) = g_{\ell+1}(z)$ by definition of the nodal interpolants. Therefore,
\begin{align*}
 (W_{\ell+1}-W_{\ell+1}^\ell)(z) = 0
 \quad\text{if }z\in(\NN_{\ell+1}\cap\Omega) \cup (\NN_\ell\cap\Gamma).
\end{align*}
Thus, it only remains to consider nodes $z\in(\NN_{\ell+1}\backslash\NN_\ell)\cap\Gamma$. In this case, $z$ is the midpoint of a refined edge $\widehat E_z$ of the unique father $\widehat T\in\TT_\ell$ with $T \subseteq \widehat T$, and there even holds $\widehat E_z\subset \Gamma$. Let $\widehat z\in\NN_\ell$ be an arbitrary endpoint of $\widehat E_z$. Then,
\begin{align*}
 (W_{\ell+1}-W_{\ell+1}^\ell)(\widehat z)
 = (g_{\ell+1}-g_\ell)(\widehat z) = 0,
\end{align*}
and the fundamental theorem of calculus, now applied for the arclength derivative, yields
\begin{align*}
 |(W_{\ell+1}- W_{\ell+1}^\ell)(z)|
 &\le \Big|\int_{\widehat z}^z (g_{\ell+1}-g_{\ell})^\prime \,d \Gamma\Big|
 \le |z-\widehat z|^{1/2}\,\norm{(g_\ell-g_{\ell+1})^\prime}{L^2(\widehat E_z)}.
\end{align*}
Combining our observations, we have now shown
\begin{align*}
 \norm{W_{\ell+1}- W_{\ell+1}^\ell}{H^1(T)}
 \lesssim \sum_{{\widehat E\in\EE_{\ell}^\Gamma}\atop{\widehat E\subset\partial\widehat T}} h_{\widehat E}^{1/2}\norm{(g_\ell-g_{\ell+1})^\prime}{L^2(\widehat E)}
 \simeq
 \Big(
 \sum_{{\widehat E\in\EE_{\ell}^\Gamma}\atop{\widehat E\subset\partial\widehat T}} h_{\widehat E}\norm{(g_\ell-g_{\ell+1})^\prime}{L^2(\widehat E)}^2
 \Big)^{1/2}.
\end{align*}
Finally, note that each edge $\widehat E\in\EE_{\ell}^\Gamma$ belongs only to one element $\widehat T\in\TT_{\ell}$. Thus, $\widehat E$ can at most be selected by all (but at most four) sons $T\in\TT_{\ell+1}$ of $\widehat T$, cf.\ Figure~\ref{fig:nvb}. This observation yields
\begin{align*}
 \norm{W_{\ell+1}- W_{\ell+1}^\ell}{H^1(\Omega)}^2
 = \sum_{T\in\TT_{\ell+1}^\Gamma}
 \norm{W_{\ell+1}- W_{\ell+1}^\ell}{H^1(T)}^2
 \lesssim \sum_{\widehat E\in\EE_{\ell}^\Gamma} h_{\widehat E}^{1/2}\norm{(g_\ell-g_{\ell+1})^\prime}{L^2(\widehat E)}^2.
\end{align*}
With the local mesh-width function $h_\ell\in L^\infty(\Gamma)$, 
this estimate reads
\begin{align*}
 \norm{W_{\ell+1}- W_{\ell+1}^\ell}{H^1(\Omega)}
 \lesssim \norm{h_\ell^{1/2}(g_\ell-g_{\ell+1})^\prime}{L^2(\Gamma)}
 \lesssim \norm{g_\ell-g_{\ell+1}}{H^{1/2}(\Gamma)},
\end{align*}
where we have used a local inverse estimate from~\cite[Proposition~3.1]{ccdpr:hypsing} in the last step.
We stress that the constant depends only on the local mesh-ratio of $\TT_{\ell+1}|_\Gamma$, whence on $\sigma(\TT_{\ell+1})$.
This concludes the proof.
\end{proof}

With the notation of Lemma~\ref{lem:stability_boundary}, we can
now formulate an additional convergence result.

\begin{proposition}\label{prop:Jell_bounded2}
The sequence of discrete solutions $U_{\ell+1,\ell}\in\AA_{\ell+1,\ell}$ satisfies
\begin{align}\label{eq:auxiliary:limit}
 |\JJ(U_{\ell+1,\ell})-\JJ(U_{\ell+1})|
 + |\JJ(U_{\ell+1,\ell})-\JJ(U_{\ell+1}^\ell)|
 \le \c{aux}\,\norm{g_{\ell+1}-g_\ell}{H^{1/2}(\Gamma)}
 \xrightarrow{\ell\to\infty}0
\end{align}
with some constant $\setc{aux}>0$ which only depends on $\TT_0$.
\end{proposition}

To prove Proposition~\ref{prop:Jell_bounded2}, we start with the observation that the energy functional $\JJ(\cdot)$ is coercive.

\begin{lemma}\label{lem:Jkoerziv}
For each sequence $w_\ell\in H^1(\Omega)$ with changing boundary data $w_\ell|_\Gamma=g_\ell$, bounded energy implies uniform boundedness in $H^1(\Omega)$, i.e.\
\begin{align}\label{eq:coercive}
 \sup\limits_{\ell\in\N}\JJ(w_\ell)<\infty
 \quad\Longrightarrow\quad
 \sup\limits_{\ell\in\N}\norm{w_\ell}{H^1(\Omega)}<\infty.
\end{align}
\end{lemma}

\begin{proof}
Note that triangle inequalities and the Friedrichs inequality yield
\begin{align*}
\begin{split}
\norm{w_\ell}{L^2(\Omega)}&\le \norm{w_\ell - \LL g_\ell}{L^2(\Omega)} + \norm{\LL g_\ell}{L^2(\Omega)}\\
& \lesssim \norm{\nabla(w_\ell - \LL g_\ell)}{L^2(\Omega)} + \norm{\LL g_\ell}{L^2(\Omega)}\\
& \lesssim \norm{\nabla w_\ell}{L^2(\Omega)} + \norm{\LL g_\ell}{H^1(\Omega)}\\
& \lesssim \norm{\nabla w_\ell}{L^2(\Omega)} + M
\end{split}
\end{align*}
with $M=\sup_{\ell}\norm{g_\ell}{H^{1/2}(\Gamma)}<\infty$,
where we have finally used continuity of $\LL$ and a~priori convergence
of $g_\ell$, cf.\ Lemma~\ref{lemma:apriori:boundary}. Consequently, we
obtain
\begin{align*}
 \norm{w_\ell}{H^1(\Omega)} \lesssim \norm{\nabla w_\ell}{L^2(\Omega)} + M
\end{align*}
and therefore
\begin{align*}
\JJ(w_\ell) &= \frac{1}{2}\norm{\nabla w_\ell}{L^2(\Omega)}^2 - (f,w_\ell)
\gtrsim (\norm{w_\ell}{H^1(\Omega)} - M)^2 - \norm{f}{L^2(\Omega)}\norm{w_\ell}{H^1(\Omega)}.
\end{align*}
The implication~\eqref{eq:coercive} is an immediate consequence.
\end{proof}

\begin{lemma}\label{lemma:boundedness}
The sequence of discrete solutions $U_\ell\in\AA_\ell$ from Algorithm~\ref{algorithm} satisfies
\begin{align}\label{eq:Jell_bounded}
 \sup_{\ell\in\N} |\JJ(U_\ell)| < \infty
 \quad\text{as well as}\quad
 \sup_{\ell\in\N} \norm{U_\ell}{H^1(\Omega)} < \infty.
\end{align}
\end{lemma}

\def\maxell{\mbox{$\max_\ell$}}
\begin{proof}
According to Lemma~\ref{lem:Jkoerziv} with $w_\ell=U_\ell$, it is sufficient to prove boundedness of the energy. To that end, we define
\begin{align*}
 V_\ell := \maxell\{P_\ell\LL g_\ell - \chi, 0\} + \chi \in \AA_\ell,
\end{align*}
where $\max_\ell$ denotes the nodewise $\max$-function, i.e. $\maxell\{v,w\}\in\SS^1(\TT_\ell)$ is defined by
\begin{align*}
\maxell\{v,w\}(z) := \max\{v(z),w(z)\}\quad \text{ for all } z\in\NN_\ell.
\end{align*}
We consider the function $W_\ell := \maxell\{P_\ell\LL g_\ell - \chi, 0\}$.
Note that, by definition, $|W_\ell(z)| \le |(P_\ell\LL g_\ell - \chi)(z)|$ for all $z\in\NN_\ell$.
According to a standard scaling argument and $P_\ell\LL g_\ell - \chi\in\SS^1(\TT_\ell)$, we infer, for all $T\in\TT_\ell$,
\begin{align*}
 \norm{W_\ell}{L^2(T)} \lesssim \norm{P_\ell\LL g_\ell - \chi}{L^2(T)}
 \quad\text{as well as}\quad
 \norm{\nabla W_\ell}{L^2(T)} \lesssim \norm{\nabla (P_\ell\LL g_\ell - \chi)}{L^2(T)}
\end{align*}
due to the nodewise estimate. From this, we obtain a constant $C>0$ with
\begin{align*}
 \norm{V_\ell}{H^1(\Omega)}
 \lesssim \norm{P_\ell\LL g_\ell - \chi}{H^1(\Omega)}
 + \norm{\chi}{H^1(\Omega)}
 \lesssim \norm{g_\ell}{H^{1/2}(\Gamma)} + \norm{\chi}{H^1(\Omega)}
 \le C<\infty
\end{align*}
since $\sup_\ell\norm{g_\ell}{H^{1/2}(\Gamma)}<\infty$ and the operator norm of $P_\ell$ depends only on $\sigma(\TT_\ell)$.
Therefore,
\begin{align*}
 \JJ(U_\ell) \le \JJ(V_\ell)
  = \frac12\,\norm{\nabla V_\ell}{L^2(\Omega)}^2 - (f,V_\ell)
  \lesssim C\,(C + \norm{f}{L^2(\Omega)}) < \infty
\end{align*}
yields $\sup\limits_{\ell\in\N}\JJ(U_\ell)<\infty$ and
concludes the proof.
\end{proof}

\begin{proof}[Proof of Proposition~\ref{prop:Jell_bounded2}]
According to Lemma~\ref{lemma:boundedness}, we have
$M := \sup_\ell\norm{U_\ell}{H^1(\Omega)}<\infty$.
With the aid of Lemma~\ref{lem:stability_boundary}, we obtain \begin{align}\label{eq:important}
\begin{split}
 \hspace*{-2mm}
 \JJ(&U_{\ell+1,\ell}) \le \JJ(U_{\ell+1}^\ell)
  = \frac{1}{2}\,\enorm{U_{\ell+1}^\ell}^2 - (f, U_{\ell+1}^\ell)
 \\&
 \le \frac{1}{2}\,\big( \enorm{U_{\ell+1}}
 + \enorm{U_{\ell+1}^\ell-U_{\ell+1}} \big)^2
 - (f,U_{\ell+1})
 + \norm{f}{L^2(\Omega)} \norm{U_{\ell+1}^\ell - U_{\ell+1}}{L^2(\Omega)}
 \\&
 = \JJ(U_{\ell+1})
 + \frac{1}{2}\enorm{U_{\ell+1}^\ell \!-\! U_{\ell+1}}^2
 + \enorm{U_{\ell+1}}\,\enorm{U_{\ell+1}^\ell \!-\! U_{\ell+1}}
 + \norm{f}{L^2(\Omega)}\norm{U_{\ell + 1}^\ell \!-\! U_{\ell+1}}{L^2(\Omega)} \\\quad&
 \le \JJ(U_{\ell+1})
 + \frac{1}{2}\,\c{boundary}^2\,\norm{g_\ell - g_{\ell+1}}{H^{1/2}(\Gamma)}^2
 + \c{boundary}\,(M + \norm{f}{L^2(\Omega)})\,
 \norm{g_{\ell+1} - g_\ell}{H^{1/2}(\Gamma)},
\end{split}
\end{align}
whence
\begin{align*}
 \JJ(U_{\ell+1,\ell}) - \JJ(U_{\ell+1})
 \lesssim \norm{g_\ell - g_{\ell+1}}{H^{1/2}(\Gamma)}^2 +
 \norm{g_{\ell+1} - g_\ell}{H^{1/2}(\Gamma)}.
\end{align*}
Estimate~\eqref{eq:important} now reveals $\sup_\ell\JJ(U_{\ell+1,\ell})<\infty$, whence $\sup_\ell \norm{U_{\ell+1,\ell}}{H^1(\Omega)} < \infty$
according to Lemma~\ref{lem:Jkoerziv} with $w_\ell=U_{\ell+1,\ell}$. Arguing as in~\eqref{eq:important}
with the ansatz $\JJ(U_{\ell+1})\le\JJ(U_{\ell+1,\ell}^{\ell+1})$ shows
\begin{align*}
 \JJ(U_{\ell+1}) - \JJ(U_{\ell+1,\ell})
 \lesssim \norm{g_\ell - g_{\ell+1}}{H^{1/2}(\Gamma)}^2 +
 \norm{g_{\ell+1} - g_\ell}{H^{1/2}(\Gamma)}.
\end{align*}
The combination of the last two inequalities yields
\begin{align*}
 |\JJ(U_{\ell+1}) - \JJ(U_{\ell+1,\ell})|
 \lesssim \norm{g_\ell - g_{\ell+1}}{H^{1/2}(\Gamma)}^2 +
 \norm{g_{\ell+1} - g_\ell}{H^{1/2}(\Gamma)}
 \xrightarrow{\ell\to\infty}0.
\end{align*}
Finally, we recall $\JJ(U_{\ell+1,\ell})\le\JJ(U_{\ell+1}^\ell)$.
Then, the above calculation also yields
\begin{align*}
 |\JJ(U_{\ell+1}^\ell)-\JJ(U_{\ell+1,\ell})|
 &= \JJ(U_{\ell+1}^\ell)-\JJ(U_{\ell+1,\ell})\\
 &= \big[\JJ(U_{\ell+1}^\ell)-\JJ(U_{\ell+1})\big]
 + \big[\JJ(U_{\ell+1})-\JJ(U_{\ell+1,\ell})\big]\\
 &\lesssim\norm{g_\ell - g_{\ell+1}}{H^{1/2}(\Gamma)}^2 +
 \norm{g_{\ell+1} - g_\ell}{H^{1/2}(\Gamma)}
 \xrightarrow{\ell\to\infty}0.
\end{align*}
This
concludes the proof.
\end{proof}

We are now ready to prove that the sequence $U_\ell\in\AA_\ell$ of discrete solutions generated by Algorithm~\ref{algorithm} indeed converges towards
the exact solution $u \in \AA$.

\subsection{Proof of Theorem~\ref{thm:convergence}}
\label{section:convergence:proof}%
With the help of Lemma~\ref{lem:stability_boundary} and Lemma~\ref{lem:rel_estimates} applied twice for $U_\ell\in\AA_{\ell+1,\ell}$, and $U_{\ell+1}^\ell \in \AA_{\ell+1,\ell}$,
we obtain
\begin{align*}
\frac12\,\enorm{U_{\ell+1}&-U_\ell}^2
\le \enorm{U_{\ell+1}-U_{\ell+1,\ell}}^2
+ \enorm{U_{\ell+1,\ell}-U_\ell}^2\\
&\le 2\,\enorm{U_{\ell+1} - U_{\ell+1}^\ell}^2
+ 2\,\enorm{U_{\ell+1}^\ell - U_{\ell+1,\ell}}^2
+ 2\,\big[\JJ(U_\ell)-\JJ(U_{\ell+1,\ell})\big]\\
&\le 2\, \c{boundary2}^2\norm{h_\ell^{1/2}(g_{\ell+1} - g_\ell)'}{L^2(\Gamma)}^2
+ 4\,\big[\JJ(U_{\ell+1}^\ell) - \JJ(U_{\ell+1,\ell})\big]\\
&\qquad+ 2\,\big[\JJ(U_\ell)-\JJ(U_{\ell+1,\ell})\big].
\end{align*}
Recall that $\Delta_{\ell+1}=\big[\JJ(U_{\ell+1})-\JJ(u_{\ell+1})\big]+\gamma\varrho_{\ell+1}^2 + \lambda\, \apx_{\ell+1}^2$.
We now use the last estimate to see
\begin{align*}
 \big[\JJ(U_{\ell+1})-\JJ(u_{\ell+1})\big]
 &= \big[\JJ(U_{\ell})-\JJ(u_{\ell})\big]
 + \big[\JJ(u_{\ell})-\JJ(u_{\ell+1})\big]\\
 &\qquad+ \big[\JJ(U_{\ell+1})-\JJ(U_{\ell+1,\ell})\big]
 +\big[\JJ(U_{\ell+1,\ell})-\JJ(U_{\ell})\big]\\
 &\le \big[\JJ(U_{\ell})-\JJ(u_{\ell})\big]
 + \big[\JJ(u_{\ell})-\JJ(u_{\ell+1})\big]\\
 &\qquad+ \big[\JJ(U_{\ell+1})-\JJ(U_{\ell+1,\ell})\big]
 + \c{boundary2}^2\norm{h_\ell^{1/2}(g_{\ell+1}-g_\ell)'}{L^2(\Gamma)}^2\\
 &\qquad- \frac14\,\enorm{U_{\ell+1}-U_\ell}^2
 +2\,\big[\JJ(U_{\ell+1}^\ell) - \JJ(U_{\ell+1,\ell})\big].
\end{align*}
Note that according to Proposition~\ref{cor:apriori:obstacle} and Proposition~\ref{prop:Jell_bounded2}
it holds that
\begin{align*}
 \alpha_\ell &:= \big|\JJ(u_{\ell})-\JJ(u_{\ell+1})\big|
 + \big|\JJ(U_{\ell+1})-\JJ(U_{\ell+1,\ell})\big|
 + 2\,\big|\JJ(U_{\ell+1}^\ell) - \JJ(U_{\ell+1,\ell})\big|
\end{align*}
tends to zero as $\ell \rightarrow \infty$.
Next, we recall that on the 1D manifold $\Gamma$, the derivative $g_\ell'$ of the nodal interpolant is the elementwise best approximation of the derivative $g'$ by piecewise constants, i.e.\
\begin{align}\label{eq:gorthogonal}
\norm{h_\ell^{1/2}(g - g_{\ell+1})'}{L^2(\Gamma)}^2 + \norm{h_\ell^{1/2}(g_{\ell+1} - g_\ell)'}{L^2(\Gamma)}^2 = \norm{h_\ell^{1/2}(g-g_\ell)'}{L^2(\Gamma)}^2
\end{align}
according to the elementwise Pythagoras theorem.
So far and with $\lambda = \c{boundary2}^2$, we thus have derived
\begin{align*}
 \Delta_{\ell+1}
 &\le \big[\JJ(U_{\ell})-\JJ(u_{\ell})\big]
 + \gamma\varrho_{\ell+1}^2
 +\lambda \apx_{\ell+1}^2
 + \lambda\,\norm{h_\ell^{1/2}(g_{\ell+1} - g_\ell)'}{L^2(\Gamma)}^2\\
 &\qquad- \frac14\,\enorm{U_{\ell+1}-U_\ell}^2
 + \alpha_\ell,\\
 &\le \big[\JJ(U_{\ell})-\JJ(u_{\ell})\big]
 + \gamma\varrho_{\ell+1}^2
 + \lambda \apx_\ell^2
  - \frac14\,\enorm{U_{\ell+1}-U_\ell}^2
 + \alpha_\ell,
\end{align*}
where we have used the Pythagoras theorem in the second step.

Next, the estimator reduction of Proposition~\ref{proposition:reduction} applies and provides constants $0<q<1$ and $\c{reduction}>0$ with
\begin{align*}
 \Delta_{\ell+1}
 &\le \big[\JJ(U_{\ell})-\JJ(u_{\ell})\big]
 + \gamma q\varrho_\ell^2
 + \lambda \apx_\ell^2
 + \Big(\gamma\c{reduction} - \frac14\Big)\,\enorm{U_{\ell+1}-U_\ell}^2
 + \alpha_\ell\\
 &\le \big[\JJ(U_{\ell})-\JJ(u_{\ell})\big]
 + \gamma q\varrho_\ell^2
 + \lambda \apx_\ell^2
 + \alpha_\ell
\end{align*}
provided the constant $0<\gamma<1$ is chosen sufficiently small.

Next, from Theorem~\ref{thm:reliability}, we infer
 $\c{reliable}^{-2}\,\big[\JJ(U_\ell)-\JJ(u_\ell)\big]
 \le \varrho_\ell^2.$

\label{dpr:bch2007}%
We plug this estimate into the last one and use the fact that $\apx_\ell^2 \le \varrho_\ell^2$ to see
\begin{align*}
 \Delta_{\ell+1}
 &\le (1-\gamma\eps\c{reliable}^{-2})\,\big[\JJ(U_{\ell})-\JJ(u_{\ell})\big]
 + \gamma(q+2\,\eps)\,\varrho_\ell^2
 +(1-\varepsilon \gamma)\lambda \apx_\ell^2
 + \alpha_\ell\\
 &\le \kappa\,\Delta_\ell + \alpha_\ell,
\end{align*}
with $\kappa:=\max\{1-\gamma\eps\c{reliable}^{-2}\,,\,q+2\,\eps, 1-\gamma\eps\}$. For $0<2\,\eps<1-q$, we obtain $0<\kappa<1$ and conclude the proof of the contraction property~\eqref{eq1:conv}.

The application of Lemma~\ref{lem:rel_estimates} for $\AA_\ell\subseteq\AA_\ell^\star$ yields $\JJ(u_\ell)\le\JJ(U_\ell)$ and hence $\Delta_\ell\ge0$. Together with $\alpha_\ell\ge0$ and $\alpha_\ell\to0$, elementary calculus yields $\Delta_\ell\to0$ as $\ell\to\infty$, cf.~\cite[Proposition~1.2]{afp2009}. From $0\le\varrho_\ell\lesssim\Delta_\ell$, we obtain $\lim_\ell\varrho_\ell=0$, and the weak reliability of $\varrho_\ell$ stated in Theorem~\ref{thm:reliability} concludes the proof of~\eqref{eq2:conv}.\qed

\begin{figure}[h]
\psfrag{numerical solution after 10 iterations}{}
\centering
\includegraphics[width=120mm]{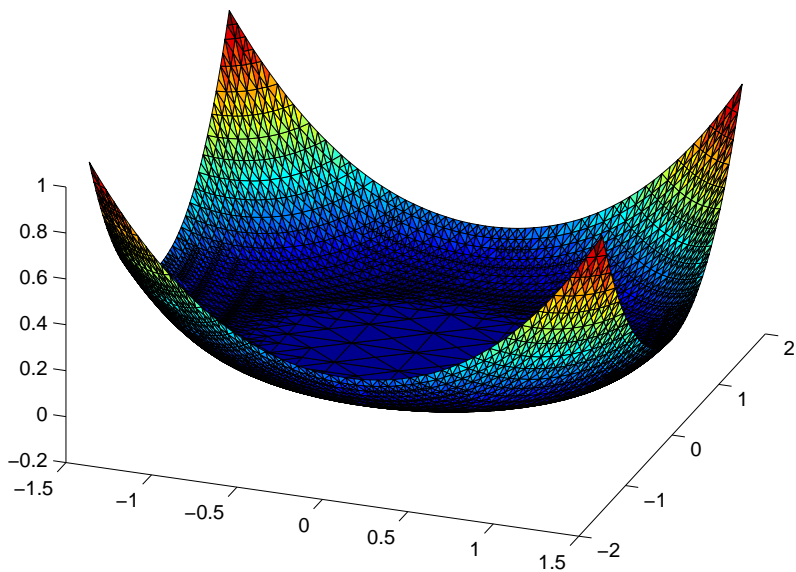}
\caption{Galerkin solution $U_{6}$ on adaptively generated mesh $\TT_{6}$ with
$N= 4.159 $ elements for $\theta=0.8$.}
\label{fig:sol}
\end{figure}

\begin{figure}[h]
\psfrag{errunif}{$\sqrt{\varepsilon_\ell}$ (unif.)}
\psfrag{etaunif}{$\varrho_\ell$ (unif.)}
\psfrag{apxunif}{$\apx_\ell$ (unif.)}
\psfrag{err08}{$\sqrt{\varepsilon_\ell}$ (adap.)}
\psfrag{eta08}{$\varrho_\ell$ (adap.)}
\psfrag{apx08}{$\apx_\ell$ (adap.)}
\psfrag{512}{$\mathcal{O}(N^{-1/2})$}
\psfrag{12}{\hspace{-4ex}$\mathcal{O}(N^{-3/4})$}
\centering
\includegraphics[width=120mm]{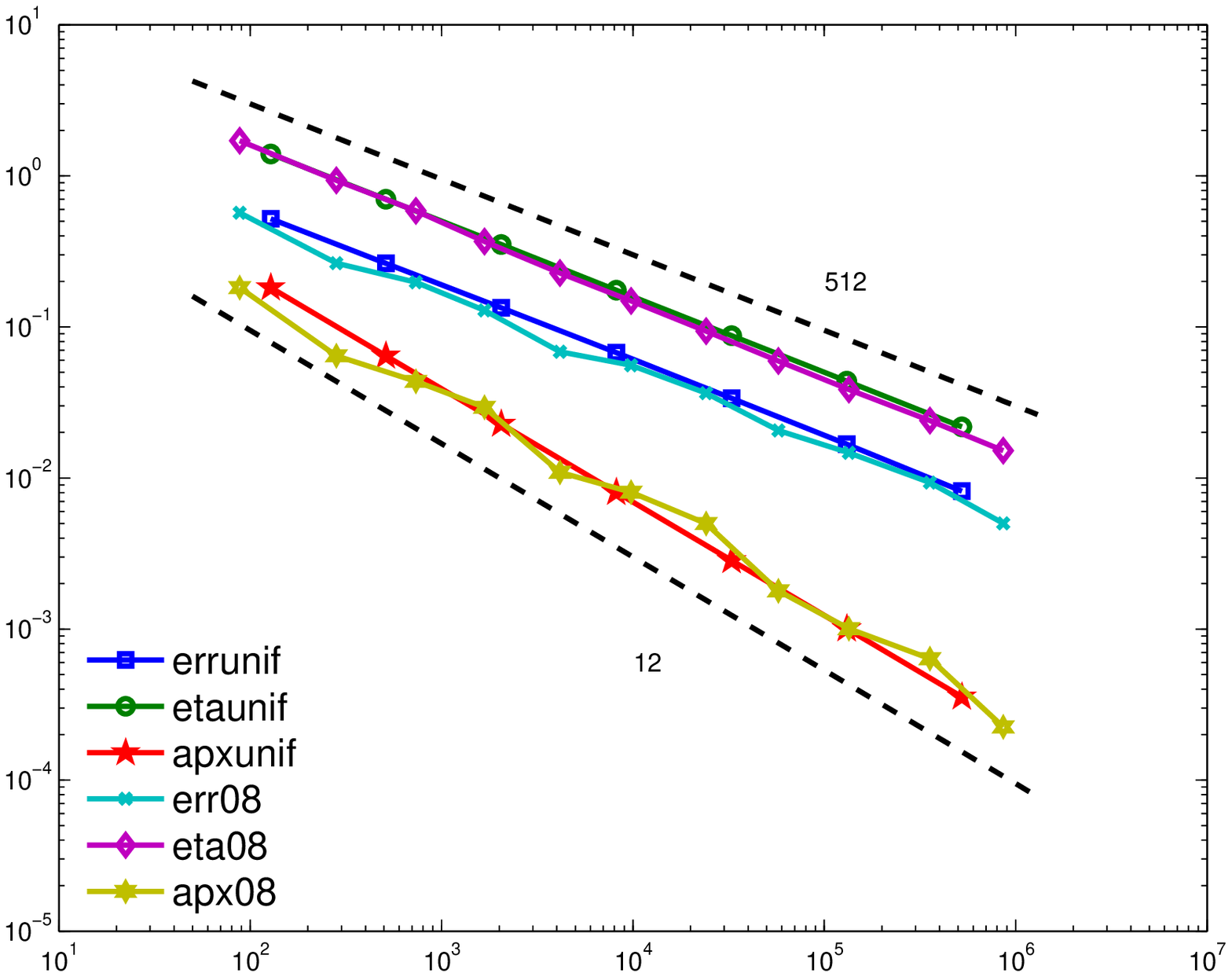}
\caption{Numerical results for uniform and adaptive mesh-refinement with $\theta = 0.8$, where $\sqrt{\eps_\ell}, \varrho_\ell,$ and $\apx_\ell$ are plotted over the number of elements $N = \#\TT_\ell$.}
\label{fig:conv1}
\end{figure}

\begin{figure}[h]
\psfrag{err04}{$\theta = 0.4$}
\psfrag{err06}{$\theta = 0.6$}
\psfrag{err08}{$\theta = 0.8$}
\psfrag{errunif}{uniform}
\psfrag{apx08}{$\apx_\ell$ (adap.)}
\psfrag{512}{$\mathcal{O}(N^{-1/2})$}
\centering
\includegraphics[width=120mm]{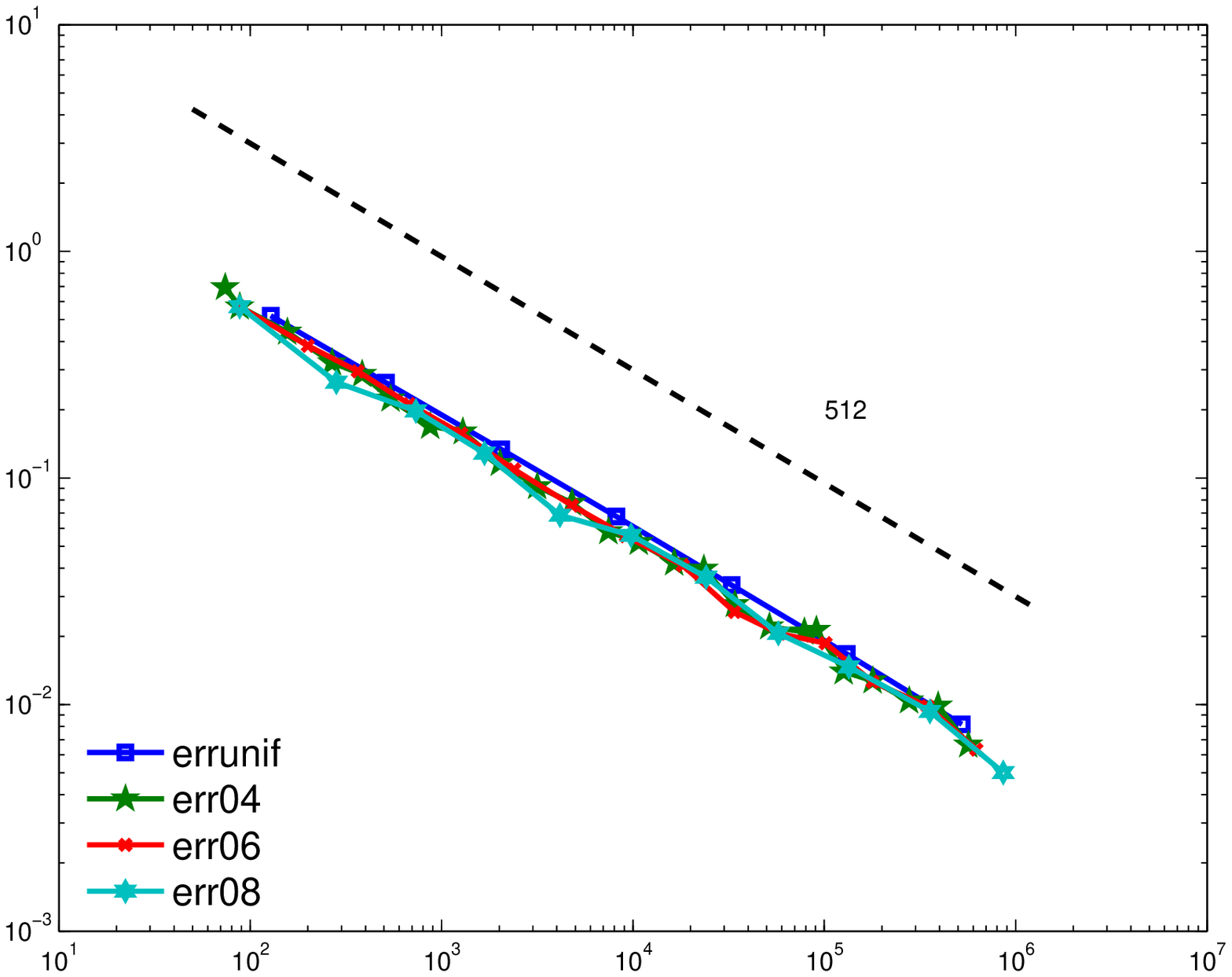}
\caption{Numerical results for $\sqrt{\eps_\ell}$ for uniform and adaptive mesh-refinement with $\theta \in \{0.4, 0.6, 0.8\}$, plotted over the number of elements $N = \#\TT_\ell$.}
\label{fig:conv2}
\end{figure}

\begin{figure}[h!]
\centering
\includegraphics[width=52mm]{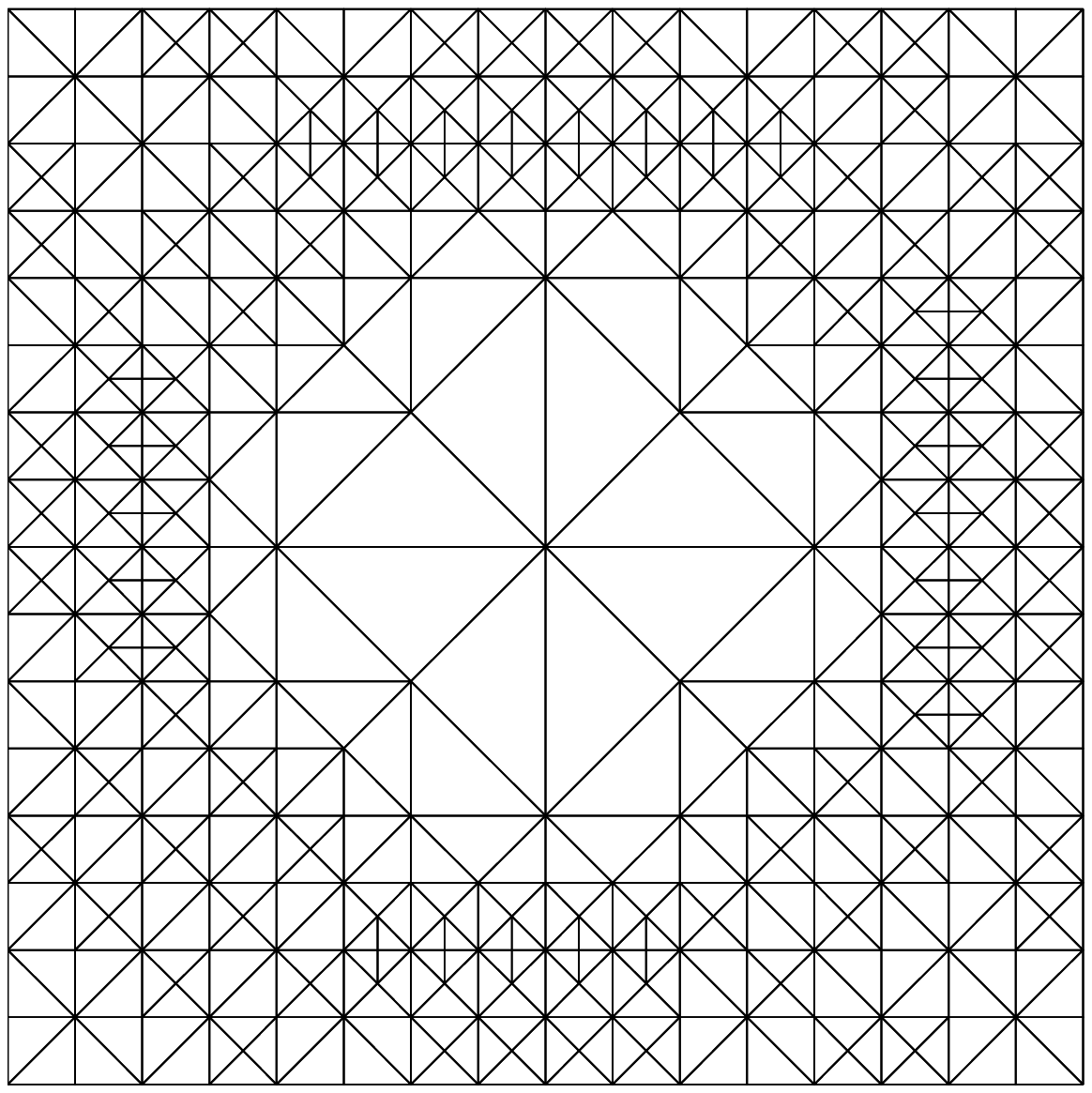}
\hspace{6ex}
\includegraphics[width=52mm]{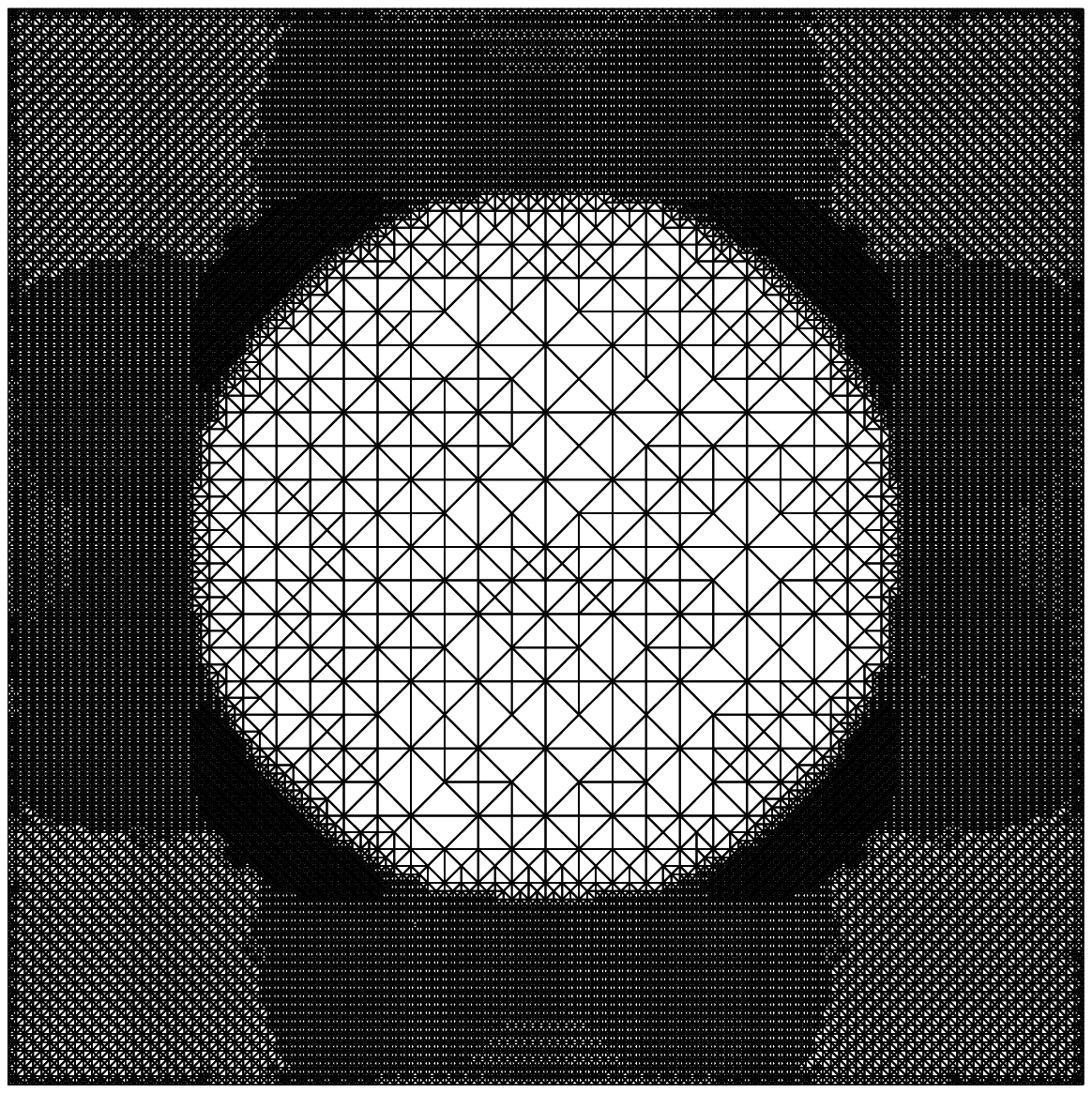}
\caption{Adaptively generated meshes $\TT_5$ (\textit{left}) and $\TT_{11}$ (\textit{right})
with $N=678$ and $N= 34.024 $ elements, respectively, for $\theta=0.6$.}
\label{fig:meshes}
\end{figure}

\section{Numerical Experiments}
\label{section:numerics}

We consider numerical examples, one of which has also been treated in~\cite{cbh2007}. The mesh in each step is adaptively generated by Algorithm~\ref{algorithm}. For the solution at each level, we use the primal-dual active set strategy from~\cite{ik2008}. The numerical results for example 1 are quite similar to those in~\cite{cbh2007}. We stress, however, that our approach includes the adaptive resolution of the Dirichlet data and, contrary to~\cite{cbh2007}, the upcoming examples are thus covered by theory.
\subsection{Example 1}
We consider the obstacle problem with constant obstacle $\chi \equiv 0$ on the square $\Omega:= (-1.5, 1.5)^2$ and a constant force $f \equiv -2$. The Dirichlet boundary data $g \in H^1(\Gamma)$ are given by the trace of the exact solution
\begin{align*}
u := \begin{cases}\frac{r^2}{2} - \ln(r) - \frac{1}{2}, & r\ge1 \\ 0, &\text{else},
\end{cases}
\end{align*}
where $r = |x|$ and $|\cdot|$ denotes the Euclidean norm on $\R^2$. The solution is visualized in Figure~\ref{fig:sol}. We compare uniform and adaptive mesh-refinement, where the adaptivity parameter $\theta$ varies between $0.4$ and $0.8$. The convergence history for uniform and adaptive refinement with $\theta = 0.8$ is plotted in Figure \ref{fig:conv1}, where the error is given in the energy functional, i.e.\
\begin{align}\label{eq:err_energyfunctional}
\varepsilon_\ell = \big|\JJ(U_\ell) - \JJ(u)\big|
\end{align}
and the Dirichlet data oscillations $\apx_\ell$ are defined by~\eqref{eq:apxE}. Note that due to Theorem~\ref{thm:convergence}, the adaptive algorithm drives the error and thus also the energy $\varepsilon_\ell$ to zero, whence it makes sense to plot these physically relevant terms. All quantities are plotted over the number of elements $N = \# \TT_\ell$ of the given triangulation. Due to high regularity of the exact solution, there are no significant benefits of adaptive refinement. We observe, however, that error and error estimator, as well as Dirichlet oscillations show optimal convergence behaviour $\OO(N^{-1/2})$ and $\OO(N^{-3/4})$ respectively. Also, the curves of $\varrho_\ell$ and $\sqrt{\varepsilon_\ell}$ are parallel, which experimentally confirms classical reliability and efficiency of the underlying estimator $\varrho_\ell$ in the sense of $\norm{u - U_\ell}{H^1(\Omega)} \simeq \varrho_\ell$.

Figure~\ref{fig:conv2} compares different values of $\theta$ and we can see that each choice of $\theta \in \{0.4, 0.6, 0.8\}$ leads to optimal convergence, i.e.\ the curves basically coincide.

Finally, Figure~\ref{fig:meshes} shows the adaptively generated meshes after 5 and 11 iterations. As expected, refinement basically takes place in the inactive zone, i.e.\ elements where the discrete solution $U_\ell$ does not touch the obstacle.

\begin{figure}[h]
\centering
\includegraphics[width=120mm]{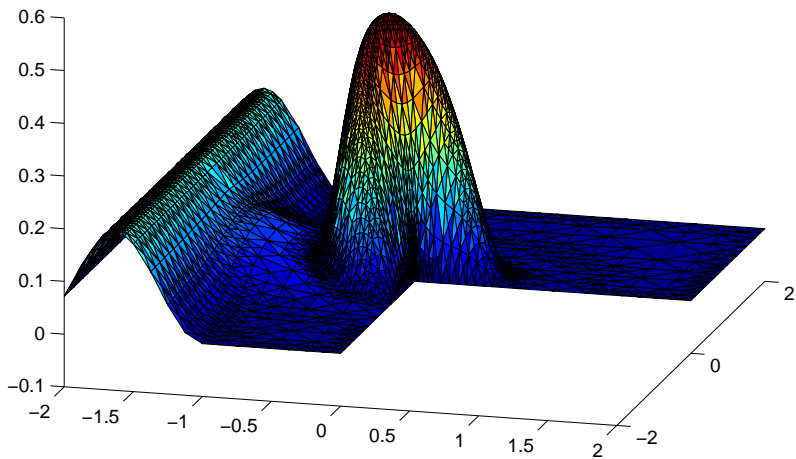}
\caption{Galerkin solution $U_{10}$ on adaptively generated mesh $\TT_{10}$ with
$N= 8.870 $ elements for $\theta=0.6$.}
\label{fig:sol_ex2}
\end{figure}

\begin{figure}[h!]
\psfrag{eta04}{\small{$\varrho_\ell$ (adap.)}}
\psfrag{err04}{\small{$\sqrt{\eps_\ell}$ (adap.)}}
\psfrag{UDell04}{\small{$\apx_\ell$ (adap.)}}
\psfrag{etaUnif}{\small{$\varrho_\ell$ (unif.)}}
\psfrag{errUnif}{\small{$\sqrt{\eps_\ell}$ (unif.)}}
\psfrag{UDellUnif}{\small{$\apx_\ell$ (unif.)}}
\psfrag{512}{\small{$\mathcal{O}(N^{-5/12})$}}
\psfrag{34}{\small{\hspace{-2ex}$\mathcal{O}(N^{-3/4})$}}
\centering
\includegraphics[width=110mm]{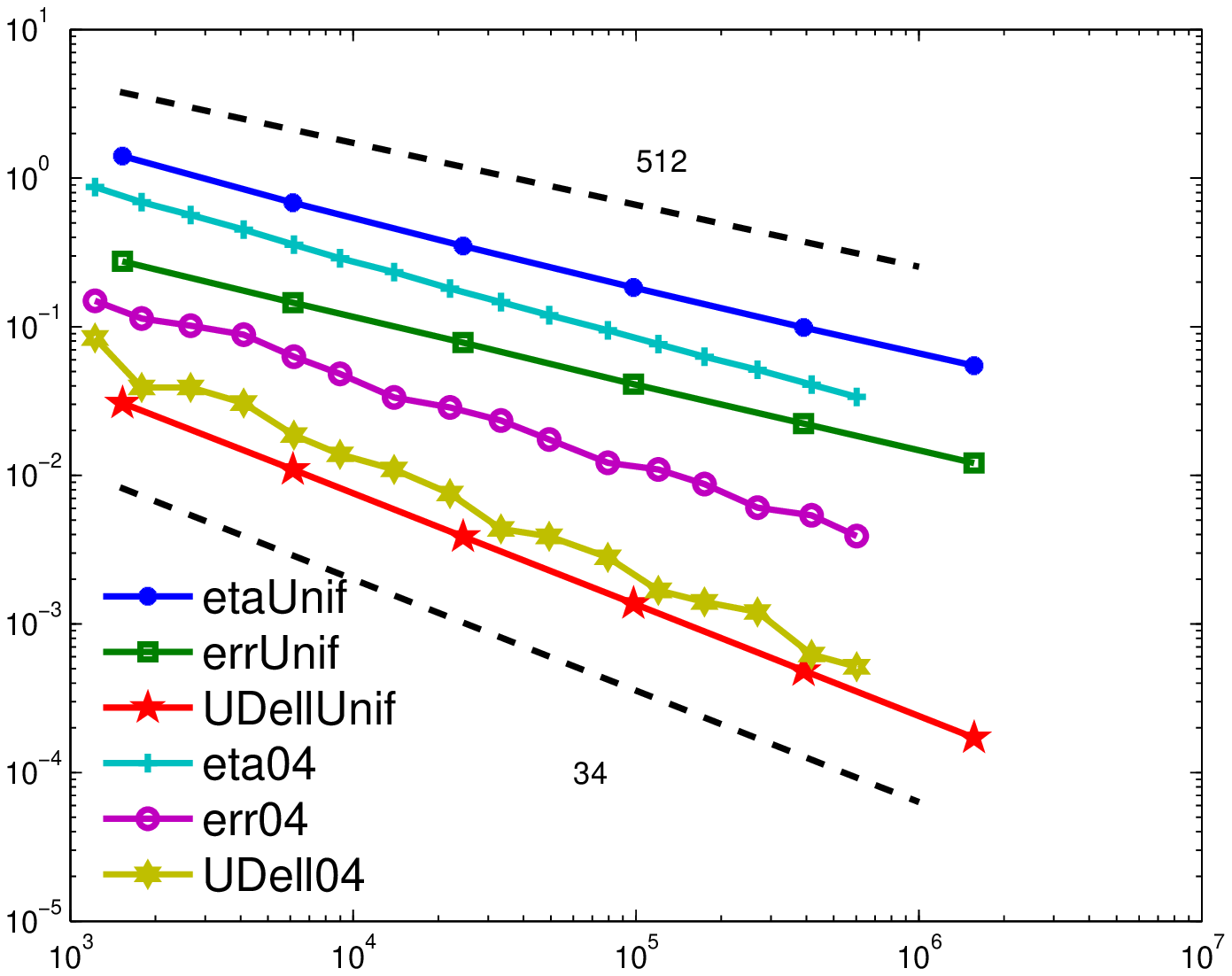}
\caption{Numerical results for uniform and adaptive mesh-refinement with $\theta = 0.4$, where $\sqrt{\eps_\ell}, \varrho_\ell,$ and $\apx_\ell$ are plotted over the number of elements $N = \#\TT_\ell$.}
\label{fig:convTotal}
\end{figure}
\begin{figure}[h]
\psfrag{04}{$\theta = 0.4$}
\psfrag{06}{$\theta = 0.6$}
\psfrag{08}{$\theta = 0.8$}
\psfrag{Unif}{uniform}
\psfrag{errunif}{$\sqrt{\eps_\ell}$ (unif.)}
\psfrag{Apxunif}{$\apx_\ell$ (unif.)}
\psfrag{512}{\hspace{-3ex}$\vspace{-4ex}\mathcal{O}(N^{-5/12})$}
\psfrag{12}{\hspace{-3ex}$\mathcal{O}(N^{-1/2})$}
\centering
\includegraphics[width=120mm]{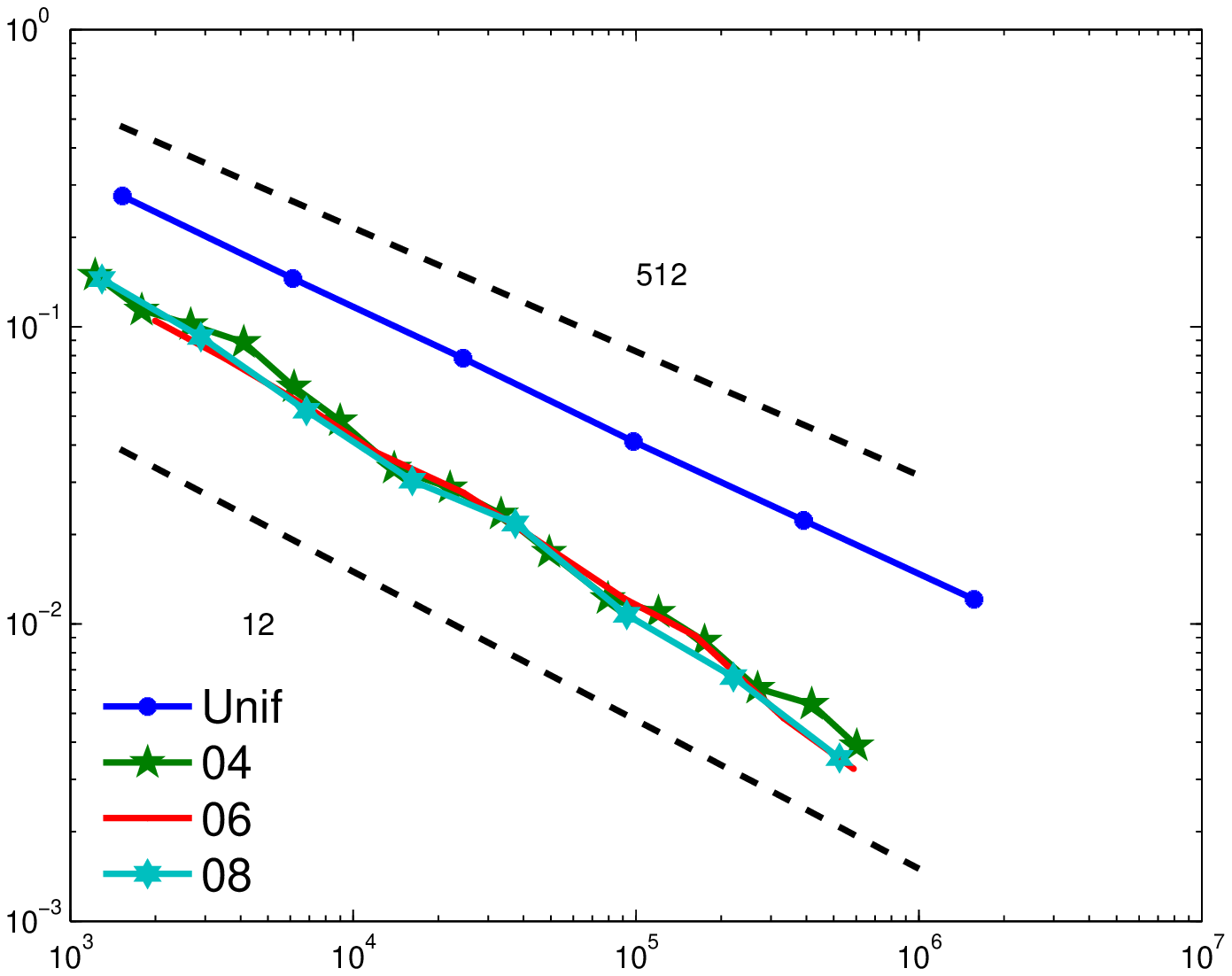}
\caption{Numerical results for $\sqrt{\eps_\ell}$ for uniform and adaptive mesh-refinement with $\theta \in \{0.4, 0.6, 0.8\}$, plotted over the number of elements $N = \#\TT_\ell$.}
\label{fig:convall}
\end{figure}

\begin{figure}[h!]
\centering
\includegraphics[width=70mm]{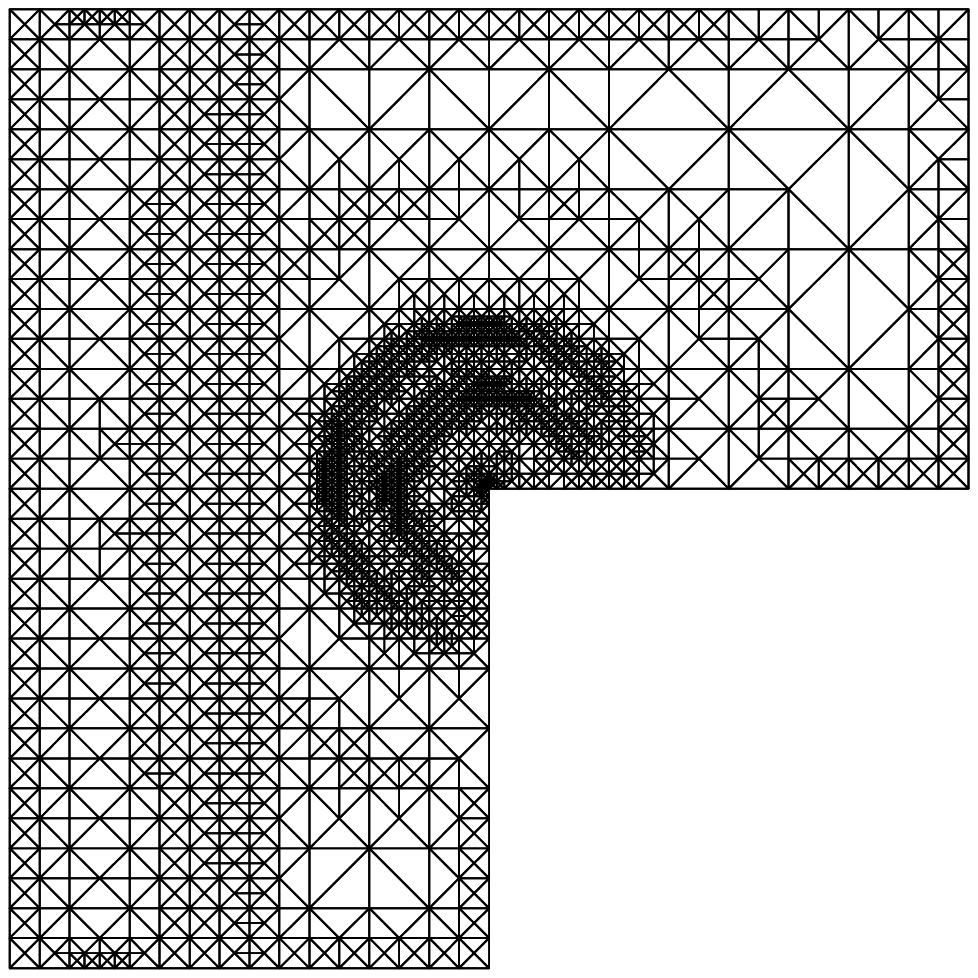}
\includegraphics[width=70mm]{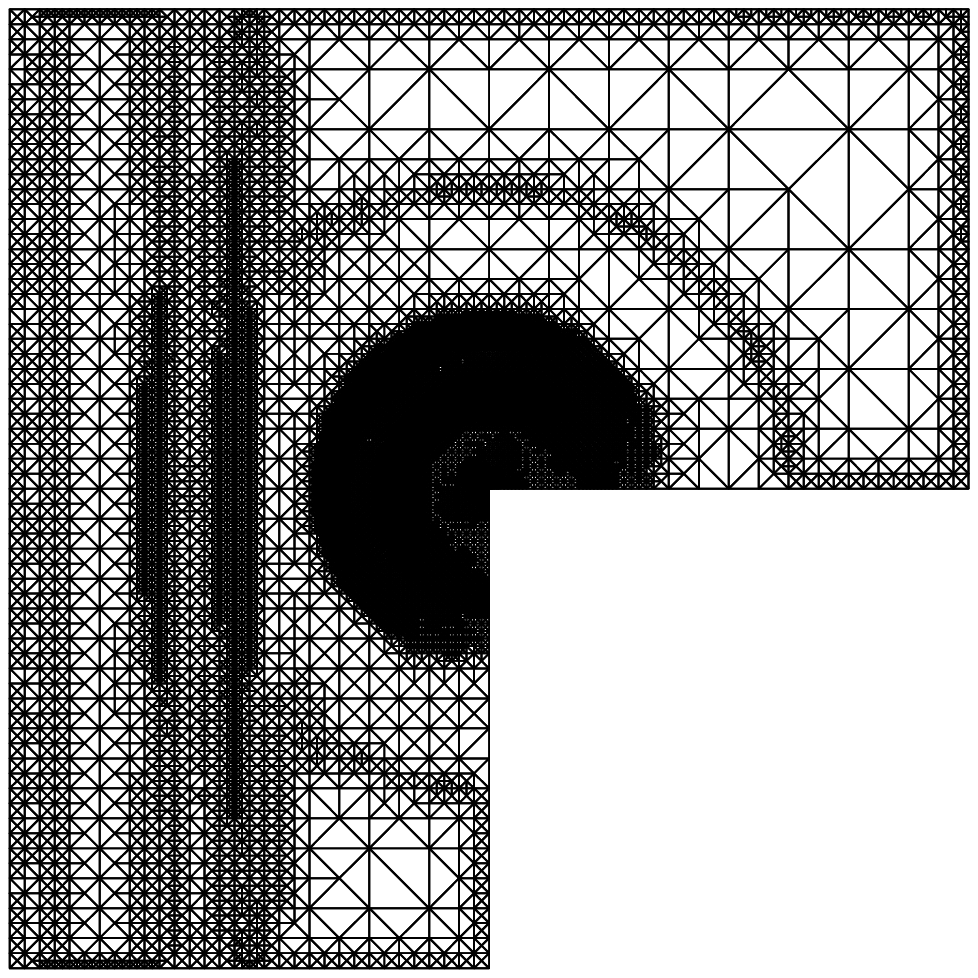}
\caption{Adaptively generated meshes $\TT_{13}$ (\textit{left}) and $\TT_{18}$ (\textit{right})
with $N=4.104$ and $N= 33.310$ elements, respectively, for $\theta=0.6$.}
\label{fig:meshes_ex2}
\end{figure}

\subsection{Example 2}
In this example, we consider the obstacle problem for a non-affine $H^2(\Omega)$ obstacle 
\begin{align*}
 \chi := \begin{cases}\frac{1}{10}\big[\sin\big(5(x+(1-\pi/10))\big)+1\big] & x < -1, \\ 0 & \text{else.}\end{cases}
\end{align*}
on the L-shaped domain $\Omega := (-2,2)^2\backslash[-2,0)$. The right hand side is given in polar coordinates by
\begin{align*}
f(r, \varphi) := 
-r^{2/3}\sin(2\varphi/3)\bigl(d/dr(\gamma_1)(r)/r + d^{\,2}/dr^{\,2}\gamma_1(r)\bigr)\\
\quad-\frac{4}{3}\,r^{-1/3}d/dr(\gamma_1)(r)\sin(2\varphi/3)-\gamma_2(r), 
\end{align*}
where $d/dr$ denotes the radial derivative. Moreover,
$\bar r := 2(r-1/4)$ and
\begin{align*}
\gamma_1(r) &:= \begin{cases}1, & \bar r < 0, \\
 -6\bar r^5 + 15 \bar r^4 -10 \bar r^3 + 1, & 0 \le \bar r < 1, \\
 0, & \bar r \ge 1, \end{cases}\\
\gamma_2(r) &:= \begin{cases}0, & r \le 5/4, \\
1, & \mbox{else.}\end{cases}
\end{align*}
The Dirichlet data $g \in H^{1/2}(\Gamma)$ are given by the trace of the obstacle $\chi$. Since the exact solution for this problem is unknown,
the Galerkin solution on a uniform mesh with approximately $N = \#\TT_\ell = 1.500.000$ elements has been used as reference solution. The non-affine obstacle
was treated by means of Proposition~\ref{prop:zero_obstacle}. Again, we compare uniform and adaptive mesh-refinement for different adaptivity pa\-ra\-me\-ters
$\theta$ between $0.4$ and $0.8$. Figure~\ref{fig:convTotal} shows the convergence history for $\theta = 0.4$ plotted over the number 
of elements $N = \#\TT_\ell$. As before, adaptive refinement leads to the optimal convergence rates $\OO(N^{-1/2})$ and $\OO(N^{-3/4})$ 
for $\sqrt{\eps_\ell}$ and $\apx_\ell$, respectively. Due to the corner singularity of the exact solution at $0$, we observe that 
uniform mesh-refinement leads to a suboptimal convergence behaviour.

Figure~\ref{fig:convall} compares the error for uniform and adaptive mesh refinement, where the adap\-ti\-vi\-ty parameter 
$\theta$ varies between $0.4$ and $0.8$. Again, we observe that each adaptive strategy leads to optimal convergence rates, whereas the 
convergence rate for uniform refinement is suboptimal.

In Figure~\ref{fig:meshes_ex2}, the adaptively generated meshes after 13 and 18 iterations are visualized. As before, refinement is 
basically restricted to the inactive zone.

\newcommand{\bibentry}[2][!]{%
\ifthenelse{\equal{#1}{!}}{\bibitem{#2}}{\bibitem{#2}}%
}

\paragraph{Acknowledgement.}
The authors M.F.\ and D.P.\ are partially funded by the Austria Science Fund (FWF) under grant P21732 \emph{Adaptive Boundary Element Methode}, which is thankfully acknowledged.
M.P.\ acknowledges partial support through the project \emph{Micromagnetic Simulations and Computational Design of Future Devices}, funded by the 
Viennese Science and Techno\-lo\-gy Fund (WWTF) under grant MA09-029.

\end{document}